\documentclass[11pt,reqno]{amsart}

\usepackage[T1]{fontenc}

\usepackage{amsmath}								
\usepackage{amssymb}
\usepackage{amsthm}
\usepackage{amscd}
\usepackage{amsfonts}
\usepackage{stmaryrd}

\usepackage{euler}

\usepackage{extarrows}

\usepackage[backref, colorlinks, linktocpage, citecolor = magenta, linkcolor = blue]{hyperref}
\usepackage{color}
\usepackage[backrefs]{amsrefs}

\usepackage{tikz}									
\usetikzlibrary{matrix}
\usetikzlibrary{patterns}
\usetikzlibrary{matrix}
\usetikzlibrary{positioning}
\usetikzlibrary{decorations.pathmorphing}
\usetikzlibrary{cd}

\usepackage{fullpage}
\usepackage{enumerate}

\newtheorem{theorem}{Theorem}[section]
\newtheorem{lemma}[theorem]{Lemma}
\newtheorem{proposition}[theorem]{Proposition}
\newtheorem{corollary}[theorem]{Corollary}

\theoremstyle{definition}
\newtheorem{definition}[theorem]{Definition}
\newtheorem{example}[theorem]{Example}

\newtheorem{remark}[theorem]{Remark}

\newcommand{\change}[1]{\null}
\newcommand{\changemore}[1]{\null}

\newcommand{\hooklongrightarrow}{\lhook\joinrel\longrightarrow}

\newcommand{\R}{\mathbb{R}}
\newcommand{\Rbar}{\overline{\mathbb{R}}}
\newcommand{\Z}{\mathbb{Z}}

\newcommand{\N}{\mathbb{N}}
\newcommand{\PP}{\mathbb{P}}

\newcommand{\A}{\mathbb{A}}

\newcommand{\GG}{\mathbb{G}}

\newcommand{\Sigmabar}{\overline{\Sigma}}
\newcommand{\sigmabar}{\overline{\sigma}}

\newcommand{\Mbar}{\overline{M}}
\newcommand{\Ubar}{\overline{U}}

\newcommand{\calCbar}{\overline{\mathcal{C}}}

\newcommand{\calC}{\mathcal{C}}

\newcommand{\calM}{\mathcal{M}}
\newcommand{\calO}{\mathcal{O}}

\newcommand{\calMbar}{\overline{\mathcal{M}}}

\newcommand{\frakp}{\mathfrak{p}}

\renewcommand{\k}{k}

\DeclareMathOperator{\Spec}{Spec}
\DeclareMathOperator{\Hom}{Hom}

\DeclareMathOperator{\trop}{trop}

\DeclareMathOperator{\id}{id}

\DeclareMathOperator{\Star}{Star}

\newcommand{\pmon}[1]{#1^{\infty}}
\newcommand{\econe}[1]{\overline{#1}}

\title[Clutching and gluing in tropical and logarithmic geometry]{Clutching and gluing in tropical and logarithmic geometry}

\author[Huszar]{Alana Huszar}
\author[Marcus]{Steffen Marcus}
\author[Ulirsch]{Martin Ulirsch}

\date{\today}

\address[Huszar]{Department of Mathematics\\
University of Michigan\\
530 Church Street\\
Ann Arbor, MI 48109\\
U.S.A.}
\email{huszara@umich.edu}

\address[Marcus]{Mathematics and Statistics\\
The College of New Jersey\\
Ewing, NJ 08628\\
U.S.A.}
\email{marcuss@tcnj.edu}

\address[Ulirsch]{Department of Mathematics\\ 
University of Michigan\\
530 Church Street\\
Ann Arbor, MI 48109\\
U.S.A.}
\email{ulirsch@umich.edu}


\begin{document}

\begin{abstract} 
The classical clutching and gluing maps between the moduli stacks of stable marked algebraic curves are not logarithmic, i.e. they do not induce morphisms over the category of logarithmic schemes, since they factor through the boundary. Using insight from tropical geometry, we enrich the category of logarithmic schemes to include so-called sub-logarithmic morphisms and show that the clutching and gluing maps are naturally sub-logarithmic. Building on the recent framework developed by Cavalieri, Chan, Wise, and the third author, we further develop a stack-theoretic counterpart of these maps in the tropical world and show that the resulting maps naturally commute with the process of tropicalization.
\end{abstract}

\maketitle

\setcounter{tocdepth}{1}


\section{Introduction}

\subsection{Main Results} In \cite{Knudsen_projectivityII}, building on the classical work of Deligne and Mumford \cite{DeligneMumford_moduliofcurves}, Knudsen has introduced the moduli stacks $\calMbar_{g,n}$ of stable algebraic curves of genus $g$ with $n$ marked points as well as a natural system of maps between these moduli spaces:
\begin{itemize}
\item the \emph{forgetful maps} $\calMbar_{g,n+1}\rightarrow\calMbar_{g,n}$ given by forgetting the $(n+1)$-st marked point (and possibly stabilizing the resulting $n$-marked curve);
\item the \emph{clutching maps} $\calMbar_{g_1,n_1+1}\times\calMbar_{g_2,n_2+1}\rightarrow \calMbar_{g_1+g_2,n_1+n_2}$ given by identifying the $(n_1+1)$-st marked point of the first curve with the $(n_2+1)$-st point of the second curve in a node; and
\item the \emph{(self-)gluing maps} $\calMbar_{g-1,n+2}\rightarrow \calMbar_{g,n}$ given by identifying the last two marked points of an $(n+2)$-marked curve in a node. 
\end{itemize}

The clutching and gluing maps factor through the boundary of $\calMbar_{g,n}$ and therefore they cannot induce morphisms between the moduli stacks $\calM_{g,n}^{log}$ of stable logarithmic curves (of genus $g$ with $n$ marked points) in the sense of \cite{Kato_logsmoothcurves}. The reason is that the monoidal coordinates that are normal to the boundary of $\calM_{g,n}^{log}$ would have to be sent to an \emph{absorbing element} $\infty$, where $\infty+p=\infty$, in the logarithmic structure of $\calM_{g,n}^{log}$ which is not allowed in the usual framework of logarithmic geometry (as developed in \cite{Kato_logstr}). Some initial examples describing this phenomenon are provided in Section~\ref{intro_examples} below.

In this note, we use insights coming from tropical geometry (see \cite[Section 8]{ACP}) to enrich the usual logarithmic structures in the sense of \cite{Kato_logstr} to so-called \emph{pointed logarithmic structures} that allow an absorbing element in the sheaf of monoidal coordinates. Using this language we may define a natural stack $\calMbar_{g,n}^{log}$ of stable logarithmic curves over the category of pointed logarithmic schemes that generalizes $\calM_{g,n}^{log}$ and that is represented by the algebraic moduli stack $\calMbar_{g,n}$ with the pointed logarithmic structure coming from its boundary divisor. The stacks $\calMbar_{g,n}^{log}$ admit natural clutching and gluing maps in the category of pointed logarithmic algebraic stacks, an important instance of what we call \emph{sub-logarithmic morphisms} (see Section~\ref{section:pointification}), whose underlying algebraic morphisms are the classical clutching and gluing maps defined in \cite{Knudsen_projectivityII}. 

Let us now give a quick outline of the contents of this article: In Section \ref{section_affine} we prove an equivalence between the category of pointed toric monoids and the category of extended rational polyhedral cones, thereby providing a dictionary between the languages used in logarithmic geometry and tropical geometry respectively.  
In Section \ref{section_pointedlogstr} we introduce pointed logarithmic structures and develop their basic geometric  properties, expanding on \cite{Kato_logstr}. In Section \ref{section_logarithmicclutching&gluing} we introduce the moduli stacks $\calMbar_{g,n}^{log}$ as well as their natural clutching and gluing maps. In Section \ref{section_tropicalclutching&gluing} we finally construct a generalization of the moduli stack $\calM_{g,n}^{trop}$ of tropical curves, as introduced in \cite{CavalieriChanUlirschWise_tropstack}, to a stack over the category of extended rational polyhedral cones as well as tropical clutching and gluing maps in this framework. Our main results, discussed in Section~\ref{section_logarithmicclutching&gluing} and  Section~\ref{section_tropicalclutching&gluing} can be summarized as follows (see Theorem~\ref{thm_logclutch&glue} and Theorem~\ref{thm_tropclutch&glue} for a formal statement).

\begin{theorem}\label{thm_main}
The classical clutching and gluing maps between products of moduli spaces of stable curves induce sub-logarithmic morphisms and naturally commute with the process of tropicalization. 
\end{theorem}

Since the $i$-th universal section of the universal family over $\calMbar_{g,n}^{log}$ is itself an instance of a clutching map, it is a straightforward corollary that the tropicalization of such a universal section provides a universal section of the tropicalization of the family. We expect such tropical universal sections will be necessary for the development of tropical analogues to the standard tautologial classes on these moduli spaces. 

In \cite{CavalieriChanUlirschWise_tropstack}*{Section 6} the authors chose to use a different route to realize clutching and gluing in logarithmic geometry. The main difference in \cite{CavalieriChanUlirschWise_tropstack} is that clutching and gluing become correspondences instead of the morphisms we study in this present paper. The approach using correspondences gives rise to \emph{generalized clutching and gluing maps} on the tropical side, as explained in \cite{ACP}*{Section 4 and 8}.

\subsection{Pointed logarithmic structures}\label{intro_examples} The central constructions of Section~\ref{section_pointedlogstr} are formulated precisely to extend the category of logarithmic schemes by allowing new ``sub-logarithmic" morphisms given by logarithmic morphisms to a logarithmic stratum of the target. Such morphisms are in general not logarithmic for the same reason that the inclusions of torus orbit closures of a toric variety are not toric morphisms: locally they are not determined by a group homomorphism of character lattices for the underlying dense tori. Our enriched category ameliorates the situation, using the dual theories of pointed toric monoids and extended rational polyhedral cones (see Section~\ref{section_affine}) to identify the \'etale local combinatorial structure of sub-logarithmic morphisms.

As a first brief example, consider the inclusion of the origin $0\in\A^1=\Spec(k[x])$ in the affine line with its toric logarithmic structure. Given the strict pullback logarithmic structure, the origin as a logarithmic scheme takes the form of the standard logarithmic point $0:=(\Spec k,\N\oplus k^\ast)$. This does admit a logarithmic map $0\in \A^1$ determined by $\N\rightarrow \N$ sending $1\mapsto 1$. However, endowing $0=(\Spec k,k^\ast)$ with its toric logarithmic structure (i.e.\ the trivial logarithmic structure) does not allow for such an inclusion map in the category of logarithmic schemes: it would be sub-logarithmic, determined by a map of pointed monoids $\N\cup\{\infty\}\rightarrow \{0,\infty\}$ sending $1\mapsto \infty$.

In one dimension higher, the inclusion $\iota:\A^1\hookrightarrow \A^2$ of the affine line $\A^1$ as the axis $y=0$ in $\A^2$ given by the dual of the homomorphism $k[x,y]\rightarrow k[t]$ sending $x\mapsto t$ and $y\mapsto 0$ is \emph{not} a logarithmic morphism when we use the logarithmic structures induced by the toric boundaries of $\A^1$ and $\A^2$ (namely, those associated to the inclusions $\N\hookrightarrow k[\N]$ and $\N^2\hookrightarrow k[\N^2]$). Indeed, such a logarithmic morphism would need to be induced by a commutative diagram 
\[
\begin{tikzcd}
k[x,y]\cong k[\N^2] \arrow[r]&  k[\N]\cong k[t]\\
\N^2\arrow[r]\arrow[u]& \N \arrow[u]
\end{tikzcd}
\]
sending $(0,1)\mapsto y\mapsto 0$, which does not occur for any monoid morphism $\N^2\rightarrow \N$. The desired morphism is instead sub-logarithmic, induced by the morphism of pointed monoids \[\N^2\cup\{\infty\}\rightarrow\N\cup\{\infty\}\] sending $(1,0)\mapsto 1$ and $(0,1)\mapsto \infty$.

Finally, given a fine and saturated logarithmic scheme $X$ that is logarithmically smooth over $k$ with $M_X$ defined in the Zariski topology on $X$, there is a natural stratification of $X$ by locally closed subsets (the logarithmic strata) that are in a one-to-one correspondence with the cones in the (rational polyhedral) cone complex associated to $X$ (see \cite[Corollary 3.5]{Ulirsch_tropcomplogreg}). With the addition of sub-logarithmic morphisms, this stratification is lifted from the category of schemes to the category of pointed logarithmic schemes, allowing for a richer theory of ``logarithmically stratified spaces" where the stratification can be combinatorially described using the local pointed monoids of the logarithmic structure. We discuss this briefly in Section~\ref{logstrat}.

\subsection{Acknowledgements} This project originated in conversations with Noah Giansiracusa and Jeffery Giansiracusa, who we would like to cordially thank. The authors also profited from discussions with Dan Abra\-movich, Renzo Cavalieri, Melody Chan, Sam Payne, Dhruv Ranganathan, and Jonathan Wise. We thank two anonymous referees for their close reading and insightful comments which helped greatly improve this paper. Parts of this project have been completed while the last two authors were visiting the Fields Institute for Research in Mathematical Sciences in Toronto and the Max Planck Institute for Mathematics in the Sciences in Leipzig; we would like to thank both institutions for their hospitality.

\section{Pointed monoids and extended rational polyhedral cones}\label{section_affine}

\subsection{Monoids and rational polyhedral cones}\label{foundations}

A \emph{monoid} $P$ is a commutative semigroup with a neutral element, usually written additively as $(P,+, 0)$. A monoid is said to be \emph{integral} if the natural map $P\rightarrow P^{gp}$ into its Grothendieck group is injective. A monoid is \emph{fine} if it is finitely generated and integral. An integral monoid is \emph{saturated} if, whenever we have $n\cdot p\in P$ for some $n>0$ and $p\in P^{gp}$, then $p\in P$. A fine and saturated monoid is \emph{toric} if it is torsion-free. Denote by $\mathbf{Mon}_{fs}$ the category of fine, saturated monoids and by $\mathbf{Mon}_{toric}$ the full subcategory of toric monoids.

An ideal $I$ of $P$ is a subset such that $p + I\subseteq I$ for all $p\in P$. An ideal $\frakp$ is called \emph{prime} if, whenever $p+q\in \frakp$ then either $p\in \frakp$ or $q\in \frakp$, or, equivalently, if the complement $P-\frakp$ is a submonoid. We write $\Spec P$ for the set of prime ideals of $P$. It carries a natural \emph{Zariski topology} generated by the basic open subsets $D(f)=\big\{\frakp\in \Spec P\big\vert f\notin \frakp\big\}$ for $f\in P$ (see \cite{Kato_toricsing} and \cite{Ulirsch}*{Section 3} for details). 

A \emph{(strictly convex) rational polyhedral cone} (cone for short) is a pair $(\sigma,N)$, consisting of a finitely generated free abelian group $N$ and a cone $\sigma\subseteq N_\R:=N\otimes \R$, that is, a finite intersection of half-spaces $\sigma=\cap_i H_i$ of the form 
\[H_i:=\Big\{u\in N_\R\Big| \langle u,m_i\rangle\geq 0\Big\} \ ,\] 
where the $m_i$ are elements of the dual lattice $M:=\Hom(N,\Z)$, so that $\sigma$ does not contain any non-trivial linear subspaces.  A morphism of rational polyhedral cones $(\sigma,N)\longrightarrow(\sigma',N')$ is a morphism $\varphi\in\Hom(N,N')$ of lattices whose canonical extention to the linear map of vector spaces $\varphi_\R:{N}_\R\longrightarrow{N'}_\R$ such that $\varphi_\R(\sigma)\subseteq \sigma'$. Abusing notation, we will denote $(\sigma,N)$ simply as $\sigma$. We denote the category of rational polyhedral cones by $\mathbf{RPC}$. 

The \emph{dual cone} $\sigma^\vee$ of $\sigma$ is given by
$$\sigma^{\vee}=\Big\{v\in M_{\R}\Big|\langle u,v\rangle\geq0\text{ for all }u\in\sigma\Big\}.$$ 
It may also be described as a finite intersection of dual half-spaces built from the $H_i$ above. A \emph{face} of a cone, i.e. an intersection $\tau=H_m\cap\sigma$, for some $m\in\sigma^{\vee}$, is written as $\tau\preceq\sigma$. The \emph{dual face} to $\tau$ is the face of $\sigma^{\vee}$ defined by $\sigma^{\vee}\cap\tau^{\perp}$, where
$$\tau^{\perp}=\Big\{u\in M_{\R}\Big|\langle u,v \rangle=0 \text{ for all } v\in\tau\Big\}.$$

We finish by recalling the following well-known Proposition \ref{prop_toricmonoid=ratpolycone}.

\begin{proposition}\label{prop_toricmonoid=ratpolycone}
\begin{enumerate}[(i)]
\item There is a natural contravariant equivalence of categories
\begin{equation*}\begin{split}
S\colon\mathbf{Mon}_{toric}&\xlongrightarrow{\sim}\mathbf{RPC}\\
P&\longmapsto \sigma_P=\Hom(P,\R_{\geq 0}) \\ 
S_\sigma=\sigma^\vee \cap M&\longmapsfrom (\sigma, N) \ .
\end{split}\end{equation*}
\item Let $P$ be a toric monoid and $\sigma=\sigma_P$ the associated cone. Given a face $\tau$ of $\sigma$, we define $\frakp_\tau:=S_\sigma-\tau^\perp$. The above equivalence induces a natural one-to-one correspondence
\begin{equation*}\begin{split}
\big\{\textrm{ faces of }\sigma \big\} & \xlongrightarrow{\sim} \Spec P\\
\tau &\longmapsto \frakp_\tau 
\end{split}\end{equation*}
such that $\tau'\preceq\tau\preceq \sigma_P$ if and only if $\frakp_{\tau'}$ specializes to $\frakp_\tau$. 
\end{enumerate}\end{proposition}

\begin{proof}
Part (i) is well-known and we leave the details of the proof to the avid reader (also see e.g. \cite{Fulton}*{Section 1.2} or \cite{Ulirsch}*{Proposition 2.2}). Note, hereby, that the lattice associated to the monoid $\sigma_P$ is given by the dual $\Hom(P^{gp},\Z)$ of $P^{gp}$. 

For Part (ii), note that $\frakp_\tau:=S_\sigma-\tau^\perp$ is a prime ideal of $S_\sigma$ since it consists of the lattice points of the complement of a face of $\sigma^\vee$. Given a prime ideal $\frakp$ of $P$, we define the associated face $\tau_\frakp$ of $\sigma_P$ to be 
\begin{equation*}
\tau_\frakp={\rm Cone}(P - \frakp)^\perp \cap \sigma_P \ .
\end{equation*}
Let $\frakp$ be a prime ideal of $P$ and $\tau=\tau_\frakp$ the associated face of $\sigma_P$. Since $S_\sigma=\sigma^\vee\cap M=P$ and $\sigma$ is a full dimensional cone inside $N_\R$, we have 
\begin{align*}
\frakp_\tau&=(\sigma^\vee- \tau^\perp)\cap M\\
&=P-({\rm Cone}(P - \frakp)^\perp \cap \sigma)^\perp \cap M\\
&=P-({\rm Cone}(P - \frakp) + \sigma^\perp)\cap M\\
&=P-(P-\frakp)=\frakp \ .
\end{align*}
Therefore we have $\frakp_{\tau_\frakp}=\frakp$. The converse verification that $\tau_{\frakp_\tau}=\tau$ is left to the reader. Finally, notice that $\tau'$ is a proper face of $\tau$ if and only if there is an element $f\in S_\sigma$ such that $f\notin \frakp_{\tau'}$ and $f\in\frakp_\tau$. This implies that $\frakp_\tau$ is in the closure of $\frakp_{\tau'}$, i.e. that $\frakp_{\tau'}$ specializes to $\frakp_\tau$.
\end{proof}

\subsection{Pointed Monoids}\label{section_pointedmonoids} 
A \emph{pointed monoid} is  a monoid $P$ containing a non-zero \emph{absorbing element}, which is to say, an element $\infty_P\in P$ satisfying \[p+\infty_P=\infty_P\] for every element $p\in P$. Such an absorbing element is unique when it exists. A morphism $f\colon P\longrightarrow Q$ between pointed monoids is a monoid homomorphism mapping the absorbing element of $P$ to the absorbing element of $Q$. Denote by $\mathbf{Mon}^{\infty}$ the category of \emph{pointed monoids}. Although $\mathbf{Mon}^{\infty}$ is a subcategory of the category of monoids $\mathbf{Mon}$, it is not a full subcategory: for instance, the map $\{0,\infty\}\rightarrow\{0,\infty\}$ taking both elements to $0$ is a monoid homomorphism, but not a morphism of pointed monoids. We refer the reader to \cite{Thesis} for further background on the theory of pointed monoids.

There is a \emph{pointification functor}
\[ 
[.]^\infty:\mathbf{Mon}\longrightarrow\mathbf{Mon}^{\infty}
\] 
that takes a monoid $P$ to $\pmon{P}:=P\sqcup\{\infty_P\}$ and a morphism $f:P\longrightarrow Q$ of monoids to the morphism $\pmon{f}:\pmon{P}\longrightarrow\pmon{Q}$ defined by $f^\infty|_P=f$, and $f^\infty(\infty_P)=\infty_Q$. The category of \emph{pointed toric monoids}, which we denote $\mathbf{Mon}_{toric}^{\infty}$, is the full subcategory of $\mathbf{Mon}^{\infty}$ generated by the image of $\mathbf{Mon}_{toric}$ via the above pointification functor $[.]^\infty$.

To set a clear notation, we will decorate all objects of $\mathbf{Mon}^{\infty}$ with the superscript infinity, and simply notate the underlying unpointed toric monoid of $\pmon{P}$ by $P$. Note that we may canonically identify the sets $\Hom_\mathbf{Mon}(P,Q^\infty)$ and $\Hom_{\mathbf{Mon}^\infty}(P^\infty,Q^\infty)$ since morphisms of pointed monoids must send absorbing elements to absorbing elements (this is not the case when the co-domain $Q$ is not already pointed). In particular, this identification is used implicitly in the definition of an extended rational polyhedral cone (see Section~\ref{section_extendedcones})

For any ideal $I\subseteq P$ there is a canonical minimal extension to an ideal $I\subseteq \pmon{P}$ given by including $\infty_P\in I$, and likewise in the other direction by removing the absorbing element. We abuse notation by using the same notation $I$ for both ideals. For each morphism $f:\pmon{P}\longrightarrow\pmon{Q}$ of pointed monoids, the set $f^{-1}(\infty_Q)$ is a prime ideal of $\pmon{P}$.

In the case that $f^{-1}(\infty_Q)=\{\infty_P\}$ consists only of the absorbing element of $\pmon{P}$, the map $f$ is simply the image via the pointification functor of the restriction $f|_{P}:P\longrightarrow Q$; we also call such a morphism of pointed monoids \emph{toric}.

Given any ideal $I\subseteq \pmon{P}$, the standard construction of a quotient monoid is given by the \emph{Rees quotient} \cite{Ree}
$$\displaystyle \pmon{P}/I=(\pmon{P}-I)\sqcup\{\infty\}.$$ 
\noindent The Rees quotient comes with the surjection
$$[.]:\pmon{P}\rightarrow\pmon{\pmon{P}/I},$$
which restricts to the identity on $\pmon{P}-I$ and takes $I$ to $\infty$.
We then see the Rees quotient has a monoid structure given by
\[ [p]+[q]=[p+q]. \] 

\noindent This construction works for ideals in both pointed and unpointed monoids, but always produces a pointed monoid as a quotient. The universal property of Rees quotients yields that any morphism $f:\pmon{P}\longrightarrow\pmon{Q}$ in $\mathbf{Mon}_{toric}^{\infty}$ admits a unique factorization 
\[
\begin{tikzcd}
\pmon{P}\arrow[r,"g"']   &  \pmon{P}/f^{-1}(\infty_Q)\arrow[r,"h"'] &   \pmon{Q}
\end{tikzcd}
\]
where $g$ is the quotient map and $h^{-1}(\infty_Q)$ consists only of the absorbing element of $\pmon{P}/f^{-1}(\infty)$. 

We now describe products and pushouts of pointed toric monoids. This is studied in general for pointed monoids in \cite{Thesis}*{Section~1}, where the smash product notation in place of the standard notations for product and amalgamated sum. We will stick with the more standard notations as no confusion should arise for our present needs.

The product $\prod_{i\in J}\pmon{P_i}$ of a family of pointed monoids $\{\pmon{P_i}\}_{i\in J}$ in the category $\mathbf{Mon}_{toric}^{\infty}$ is given by first taking their product in the category of monoids, and then taking the Rees quotient by the ideal \[I=\big\{(p_i)_{i\in J}\big|p_i=\infty \text{ for some }i\in J \big\}.\]  In this way, the product of pointed monoids produces the pointification of their underlying product: 
\[\prod_{i\in J} \pmon{P_i} = \left(\prod_{i\in J}P_i\right)^{\infty}.\]
Given two toric morphisms $\pmon{P}\longrightarrow\pmon{Q}$ and $\pmon{P}\longrightarrow\pmon{Q'}$ then, it is straightforward to see that their pushout 
\[
\begin{tikzcd}
\pmon{P}\arrow[r]\arrow[d]&\pmon{Q}\arrow[d]\\
\pmon{Q'}\arrow[r]&\pmon{Q}\oplus_{\pmon{P}}\pmon{Q'}
\end{tikzcd}
\]
is given by the pointification of the underlying amalgamated sum. In other words, we have 
\begin{equation*}
\pmon{Q}\oplus_{\pmon{P}}\pmon{Q'}=\left(Q\oplus_{P}Q'\right)^\infty \ .
\end{equation*}
More generally, we have the following lemma collecting some necessary facts from \cite{Thesis} which we prove here for the sake of the reader.

\begin{lemma}\label{pushoutlemma}
Let $f:\pmon{P}\longrightarrow\pmon{Q}$ and $g:\pmon{P}\longrightarrow \pmon{Q'}$ be morphisms in $\mathbf{Mon}_{toric}^{\infty}$. 
\begin{enumerate}
\item The pushout of $f$ and $g$, denoted $\pmon{Q}\oplus_{\pmon{P}}\pmon{Q'}$, exists in $\mathbf{Mon}_{toric}^{\infty}$. \item If $\pmon{Q'}=\pmon{P}/I$ is a Rees quotient, then $\pmon{Q}\oplus_{\pmon{P}}\pmon{P}/I=\pmon{Q}/(f(I)+\pmon{Q})$.
\item If both $\pmon{Q}=\pmon{P}/J$ and $\pmon{Q'}=\pmon{P}/I$ are Rees quotients, then $\pmon{P}/J\oplus_{\pmon{P}}\pmon{P}/I=\pmon{P}/I\cup J$.
\end{enumerate}
\end{lemma}

\begin{proof} We begin with (2) and note that (3) follows immediately from (2) since $(f(I)+\pmon{P}/J)$ the smallest ideal in $\pmon{P}$ containing both $I$ and $J$ is simply $I\cup J$.
Now, let $(f(I)+\pmon{Q})$ be the smallest ideal in $\pmon{Q}$ containing the image of $I$ and consider the Rees quotient $h:\pmon{Q}\longrightarrow \pmon{Q}/(f(I)+\pmon{Q})$. Since the preimage of $\infty$ through the composition $h\circ g$ is an ideal in $\pmon{P}$ that must contain $I$ by construction, the morphism $g$ descends to a toric morphism $g_I:\pmon{P}/I\longrightarrow \pmon{Q}/(f(I)+\pmon{Q})$ fitting into a commutative diagram
\begin{equation}
\begin{tikzcd}\label{diag1}
\pmon{P}\arrow[r,"g"]\arrow[d,"f"']&\pmon{Q}\arrow[d,"h"]\\
\pmon{P}/I\arrow[r,"g_I"']&\pmon{Q}/(f(I)+\pmon{Q}).\\
\end{tikzcd}
\end{equation}

Consider any pointed toric monoid $\pmon{M}$ fitting into a commutative diagram
\begin{equation} \label{diag2}
\begin{tikzcd}
\pmon{P}\arrow[r,"g"]\arrow[d,"f"']&\pmon{Q}\arrow[d,"a"]\\
\pmon{P}/I\arrow[r,"b"']&M.
\end{tikzcd}
\end{equation}
By commutativity of this diagram, we must have an inclusion of ideals \[(f(I)+\pmon{Q}) \subseteq a^{-1}(\infty_M)\] in $\pmon{Q}$. The unique factorization of $a$ in $\mathbf{Mon}_{toric}^{\infty}$ may then be expressed as a composition 
\[
\begin{tikzcd}
\pmon{Q}\arrow[r,"h"] &\pmon{Q}/(f(I)+\pmon{Q})\arrow[r]&\pmon{Q}/ a^{-1}(\infty_M) \arrow[r] &\pmon{M} 
\end{tikzcd}
\]
of two Rees quotients followed by a toric morphism. The composition of the second and third map in this sequence gives our desired morphism $\pmon{Q}/(f(I)+\pmon{Q})\longrightarrow \pmon{M}$ fitting into the larger commutative diagram combining Diagrams~(\ref{diag1}) and (\ref{diag2}). Any other such map $\phi$ must factor first through the map $\pmon{Q}/(f(I)+\pmon{Q})\longrightarrow \pmon{Q}/a^{-1}(\infty)$, and precompose with $h$ to commute with $a$, so our map is unique.

To prove (1), we have that $f$ and $g$ factor uniquely as
\[
\begin{tikzcd}
\pmon{P}\arrow[r] & \pmon{P}/f^{-1}(\infty_Q)\arrow[r]& \pmon{Q}
\end{tikzcd}
\]
and 
\[
\begin{tikzcd}
\pmon{P}\arrow[r] & \pmon{P}/g^{-1}(\infty_{Q'})\arrow[r]& \pmon{Q'}
\end{tikzcd}
\]
where the first morphisms are Rees quotients and the second are toric. Using (2) and (3) we can construct a larger diagram
\[
\begin{tikzcd}
\pmon{P} \arrow[r]\arrow[d]& \pmon{P}/f^{-1}(\infty_Q)\arrow[r]\arrow[d] & \pmon{Q}\arrow[d]\\
\pmon{P}/g^{-1}(\infty_{Q'})\arrow[r]\arrow[d] & \pmon{P}/f^{-1}(\infty_Q)\cup g^{-1}(\infty_Q')\arrow[r]\arrow[d] & \pmon{Q}/(f(g^{-1}(\infty_{Q'}))+Q))\\
\pmon{Q'}\arrow[r]& \pmon{Q'}/(g(f^{-1}(\infty_{Q}))+Q')) &
\end{tikzcd}
\]
where all the squares are pushouts and the bottom-right two morphisms are toric. Thus the final desired pushout is the pushout of these final two toric morphisms, i.e.\ the pointification of the pushout of the underlying morphisms in the category of toric monoids.
\end{proof}

\subsection{Extended rational polyhedral cones}\label{section_extendedcones}

Following \cite{Thuillier} and \cite[Section~2]{ACP}, a rational polyhedral cone $\sigma=\Hom_\mathbf{Mon}(P,\R_{\geq0})$ has a canonical compactification given by 
\begin{equation*}
\overline{\sigma}=\Hom_\mathbf{Mon}(P,\R^\infty_{\geq0})=\Hom_{\mathbf{Mon}^\infty}(\pmon{P},\R^\infty_{\geq0})
\end{equation*}  
where $\R^\infty_{\geq0}=\R_{\geq0}\cup\{\infty\}$ is the extended non-negative real half-line, a pointed monoid under addition. We call $\overline{\sigma}$ the associated \emph{extended rational polyhedral cone} (extended cone for short).

Given a face $\tau$ of $\sigma$, let $N_\tau$ denote the sublattice of $N$ spanned by the points in $\tau\cap N$ and set $N(\tau)=N/N_\tau$. Letting $[.]:N\rightarrow N(\tau)$ denote the quotient map, we define the \emph{cone quotient}, $\sigma/\tau$, as $[\sigma]$. The faces of $\sigma$ containing $\tau$ correspond to the faces of $\sigma/\tau$ via $\tau'\mapsto[\tau']$. 

We denote by $\phi_\tau:\sigma\longrightarrow\sigma/\tau$ the corresponding quotient map of cones. Given any sequence of intermediate faces $\tau\preceq\tau'\preceq\sigma$ the quotient map factors uniquely as a composition 

\[
\begin{tikzcd}
\phi_{\tau'}\colon\sigma\arrow[r,"\phi_\tau"]& \sigma/\tau\arrow[r]&\sigma/\tau'.
\end{tikzcd}
\]

The extended boundary $\sigmabar-\sigma$ is a union of locally closed subsets, the \emph{faces at infinity}, determined uniquely by inclusions of faces $\tau'\preceq\tau$ in $\sigma$ (see e.g. \cite[Section 3]{Payne}, \cite{Rabinoff}*{Proposition 3.4}, and \cite{Ulirsch}*{Section 3.4}). We phrase this observation, adapted to our situation, as follows.

\begin{proposition}\label{prop_extendedcones} Let $\sigma$ be a rational polyhedral cone and $\tau\preceq\sigma$ one of its faces. 

\begin{enumerate}[(i)]
\item Given a face $\tau$ of $\sigma$, we define $i_{\tau}\colon \sigma/\tau\rightarrow \sigmabar$ by associating to an element in $\sigma/\tau$ represented by $v\in\sigma$ the unique element $i_\tau(v)\in\sigmabar=\Hom_{\mathbf{Mon}}(S_\sigma,\R^\infty_{\geq 0})$ such that
\begin{equation*}
u\longmapsto
\left\{\begin{array}{cl}
\langle v,u\rangle &\quad\textrm{ if } u\in \tau^\perp\cap S_\sigma  \\
\infty &\quad\textrm{ else}\\ \end{array}\right .
\end{equation*}
for all  $u\in S_\sigma$. The maps $i_\tau$ are well-defined and induce a stratification
\begin{equation*}
i\colon \bigsqcup_{\tau\preceq \sigma}\sigma/\tau\xlongrightarrow{\sim} \sigmabar
\end{equation*}
by locally closed subsets. 
\item Given a face $\tau$ of $\sigma$, the closure of $i(\sigma/\tau)$ in $\sigmabar$ is equal to the canonical compactification $\overline{\sigma/\tau}$ of $\sigma/\tau$. In fact, we have a commutative diagram
\begin{equation*}
\begin{tikzcd}
\bigsqcup_{\tau\preceq\tau'\preceq\sigma}\sigma/\tau' \arrow[r,"\sim"] \arrow[d,"\subseteq"]&\overline{\sigma/\tau}\arrow[d,"\overline{i}_\tau"]\\
\bigsqcup_{\tau'\preceq\sigma}\sigma/\tau' \arrow[r,"\sim"]&\sigmabar
\end{tikzcd}
\end{equation*}
\end{enumerate}
\end{proposition}

\begin{proof}
Part (i) follows immediately from \cite{Rabinoff}*{Proposition 3.4 (i)} by noticing that $\sigma/\tau$ is nothing but the image of $\sigma$ in $\Hom_{\mathbf{Mon}}(S_\sigma,\R^\infty)$ under the natural inclusion $\Hom_{\mathbf{Mon}}(S_\sigma,\R^\infty_{\geq 0})\hookrightarrow \Hom_{\mathbf{Mon}}(S_\sigma,\R^\infty)$. For part (ii) notice $\overline{\sigma/\tau}$ consists of the strata $\sigma/\tau'$ of $\sigmabar$ for which $\tau\preceq \tau'$.
\end{proof}

We refer to $\overline{i}_\tau(\overline{\sigma/\tau})\subseteq\sigmabar$ as the \emph{maximal face at infinity} of $\sigmabar$ associated to $\tau\preceq\sigma$. A \emph{morphism of extended cones} is given by a continuous map of topological spaces, $\overline f:\sigmabar\longrightarrow\sigmabar'$ so that the restriction map $f:=\overline f|_{\sigma}$ is a cone morphism into some maximal face at infinity of $\sigmabar'$. The category of extended rational polyhedral cones will be denoted by $\mathbf{RPC}^\infty$.

A morphism $\overline f:\sigmabar\longrightarrow\sigmabar'$ of extended cones is completely determined by the data of a pair $(\tau', f:\sigma\longrightarrow \sigma'/\tau')$ where $\tau'$ is a face of $\sigma'$ and $f$ is a morphism of cones: $i(\overline{\sigma'/\tau'})$ is the maximal face of $\sigma'$ containing the image of $\overline f$, the map $f:\sigma\longrightarrow \sigma'/\tau'$ is the underlying map of cones extending uniquely to $\overline f$ by continuity, and $\overline f$ is given by the composition
\[
\begin{tikzcd}
\sigmabar\arrow[rr,"\overline f"]\arrow[dr,"f"']   &   &   \sigmabar' \\
                    & \econe{\sigma'/\tau'} \arrow[ur,hook,"i_{\tau'}"']. & 
\end{tikzcd}
\] 
 The inclusion map $i_{\tau'}$ identifying the extended cone $\econe{\sigma'/\tau'}$ with a subset of $\Hom_{\mathbf{Mon}^\infty}(\pmon{S_{\sigma'}},\R^\infty_{\geq0})$ is determined by sending an element $u\in\sigma'/\tau'$ to the homomorphism
\begin{equation*}
p\longmapsto\left\{
  \begin{array}{l l}
   \langle u,p\rangle & \quad \text{if } p\in (\tau'^{\perp}\cap \sigma'^\vee) \cap S_{\sigma'}\\
   \infty & \quad \text{else}\\
  \end{array} \right .
\end{equation*}
and extending by continuity (see \cite[Proposition~3.4]{Rabinoff} or \cite[Section~3]{Payne}). When $\overline f=(0,f:\sigma\longrightarrow \sigma')$ then $\overline f$ is the canonical extension of the morphism of cones $f$ by continuity; we call such a morphism of extended cones \emph{toric}.

Given morphisms of extended cones $\overline f:\sigmabar_1\longrightarrow\sigmabar_2$ and $\overline g:\sigmabar_2\longrightarrow\sigmabar_3$, we now describe their composition. The morphisms $\overline f$ and $\overline g$ are determined respectively by pairs $(\tau, f:\sigma_1\to\sigma_2/\tau)$ and $(\upsilon, g:\sigma_2\to\sigma_3/\upsilon)$ with corresponding quotient maps for the cone quotients $\phi_{\tau}:\sigma_2\to\sigma_2/\tau$ and $\phi_{\upsilon}:\sigma_3\to\sigma_3/\upsilon$. 

\begin{lemma}\label{conefactoring}
Let $\sigma\longrightarrow \sigma'$ be a morphism of cones and $\tau\preceq \sigma$. Then $f$ factors through $\phi_\tau:\sigma\longrightarrow\sigma/\tau$ if and only if $f(\tau)=0$. When such a factorization exists, it is unique.
\end{lemma}

\begin{proof}
This follows immediately from the properties of quotients of vector spaces.
\end{proof}

By Lemma~\ref{conefactoring}, there exists a unique morphism of cones $h:\sigma_2/\tau\to\sigma_3/\omega$ fitting into the commutative diagram 
\[
\begin{tikzcd}
&\sigma_2\arrow[r,"g"]\arrow[d,"\phi_{\tau_2}"]& \sigma_3/\tau_3\arrow[d]\\
\sigma_1\arrow[r,"f"]  & \sigma_2/\tau_2\arrow[r,dashed,"h"] & \sigma_3/\omega
\end{tikzcd}
\]
if and only if $\omega$ is a face of $\sigma_3$ containing both $\upsilon$ and $g(\tau)$. Continuity of $\overline g$ ensures that the composition $\overline g \circ \overline f:=(\omega,h \circ f:\sigma_1\to\sigma_3/\omega)$ lands in the cone $\sigma_3/\omega$ for $\omega$ the \emph{smallest} face of $\sigma_3$ containing both $\upsilon$ and $g(\tau)$.  This is made precise in the following lemma. 

\begin{lemma}
Let $\overline f=(0,\sigma\longrightarrow\sigma')$ be a toric morphism of extended cones and let $\sigma/\tau$ be the maximal face of $\sigmabar$ at infinity corresponding to $\tau\preceq \sigma$. Then $\overline f(\sigma/\tau)\subseteq \sigma'/\omega$ where $\omega$ is the smallest face of $\sigma'$ containing $f(\tau)$. 
\end{lemma}

\begin{proof}
The image of $\sigma/\tau$ through $\overline f$ must land in some maximal face of $\sigma'$ at infinity. Denote this face at infinity by $\sigma'/\omega$ and assume for contradiction that $\omega$ is not the smallest face of $\sigma'$ containing $f(\tau)$. Then there is some intermediary face $\gamma\preceq\sigma'$ with $\gamma\prec\omega$ and $f(\tau)$ a subcone of $\gamma$. We then have that $\sigma'/\gamma$ is a maximal face of $\sigma'/\omega$ at infinity, with $f(\tau)\subset \gamma$ but $\overline f(\sigma/\tau) \subset \sigma'/\omega$. Let $p\in \tau$ be a vector in the maximal face of $\tau$ and consider the limit point $\displaystyle \overline p=\lim_{t\to \infty} tp$. In the topology of $\sigma$, we must have $\overline p\in \sigma/\tau$. Notice, however, that $f(tp)\in\gamma$ for any $t\in\R_{\geq0}$, but $\overline f(\overline p)\in\sigma'/\omega$, contradicting the continuity of $\overline f$.
\end{proof}

\subsection{Pointed monoids and extended cones} \label{section_pointedmonoid=extendedcone}
In this section we are going to show that (the opposite of) the category of pointed toric monoids is equivalent to the category of extended rational polyhedral cones. 

\begin{theorem}
There is a natural contravariant equivalence of categories
\begin{equation*}\begin{split}
\overline{S}:\mathbf{Mon}_{toric}^\infty & \xlongrightarrow{\sim} \mathbf{RPC}^\infty \\ 
P^\infty&\longmapsto \sigmabar_P=\Hom_{\mathbf{Mon}}(P,\R^\infty_{\geq 0}).
\end{split}\end{equation*}
\end{theorem}

\begin{proof}
Consider the association \[\overline S:\mathbf{Mon}_{toric}^{\infty,op}\longrightarrow\mathbf{RPC}^\infty\] that takes a pointed monoid $\pmon{P}$ to the extended cone 
\[  \overline S(\pmon{P})=\econe{\sigma}_P:=\Hom_{\mathbf{Mon}^\infty}(\pmon{P},\R^\infty_{\geq0}),\]
 and a morphism $f:\pmon{P}\longrightarrow\pmon{Q}$ to the morphism defined by pre-composition: 
\begin{align*} 
\overline S(f):&\ \econe{\sigma}_Q\longrightarrow\econe{\sigma}_P\\ 
&\ \ \alpha\longmapsto\alpha\circ f.
\end{align*}

Pre-composition is functorial; thus we are left to check that these constructions indeed send a morphism of pointed monoids to morphisms in $\mathbf{RPC}^\infty$ respectively. A morphism $f:\pmon{P}\longrightarrow\pmon{Q}$ factors uniquely as 
\[
\begin{tikzcd}
\pmon{P}\arrow[rr,"f"]\arrow[dr,"g"']   &   &   \pmon{Q}\\
                    & \pmon{P}/f^{-1}(\infty_Q)\arrow[ur,"h"']. & 
\end{tikzcd}
\]
Since $\overline{S}$ restricts classically to a contravariant equivalence between $\mathbf{Mon}_{toric}$ and $\mathbf{RPC}$, it sends toric morphisms to toric morphisms. By Proposition~\ref{prop_toricmonoid=ratpolycone} (ii) we know $f^{-1}(\infty_Q)=\frakp_\tau$ for $\tau={\rm Cone}(P^\circ - f^{-1}(\infty_Q))^\perp \cap \sigma_P$, so the Theorem follows from the two following Lemmas.
\end{proof}

\begin{lemma}\label{imageofreese}
Let $\tau\preccurlyeq \sigma_P$ be a face and set $r:\pmon{P}\longrightarrow \pmon{P}/\frakp_\tau$ to be the Rees quotient by $\frakp_\tau$. Then $\overline S(r)=i_\tau$, i.e. $\overline S$ sends $r$ to the canonical inclusion of the extended face at infinity $\econe{\sigma/\tau}$ in $\sigmabar_P$. 
\end{lemma}
\begin{proof}
By definition, $\overline S(r)$ results in the morphism 
\[
\begin{tikzcd}
\Hom_{\mathbf{Mon}^\infty}({\pmon{P}/\frakp_\tau},\R^\infty_{\geq0})\arrow[r,"\overline S(r)"]  & \Hom_{\mathbf{Mon}^\infty}({\pmon{P}}\R^\infty_{\geq0})\\
\alpha\arrow[r,mapsto]&\alpha\circ r.
\end{tikzcd}
\]
It is straightforward to see this is an injective morphism of monoids. 
We show that $\overline S(r)$ has image $\overline{\sigma/\tau}$. Certainly, given any map
\[\alpha:\pmon{P}/\frakp_\tau\longrightarrow\R^\infty_{\geq0}\] this map must send elements of $\frakp_\tau$ to infinity, so the image of $\overline S(r)$ is contained in $\overline{\sigma/\tau}$. Indeed it is onto: Let \[c:\pmon{P}\longrightarrow\R^\infty_{\geq0}\] be given with $c\in\econe{\sigma/\tau}$. Then 
\[c: u\longmapsto
\begin{cases}
\langle u,p \rangle & \text{for } p\in(\sigma^\vee\cap\tau^\perp)\cap P\\
\infty & \text{otherwise}
\end{cases}
\]
by definition and the map $\pmon{P}/\frakp_\tau\longrightarrow \R^\infty_{\geq0}$ given by
\[ [x]\longmapsto
\begin{cases}
\langle x,p \rangle & \text{for } p\in(\sigma^\vee\cap\tau^\perp)\cap P\\
\infty & \text{otherwise}
\end{cases}
\]
is the preimage of $c$.
\end{proof}

\begin{lemma}
The functor $\overline S$ is full, faithful, and essentially surjective. 
\end{lemma} 
\begin{proof}
Essentially surjective is clear since given any extended cone $\sigmabar$ we have $\overline S(S_\sigma^\infty)=\sigmabar$.

Let $\pmon{P}$ and $\pmon{Q}$ be given. To show $\overline S$ is full, we will find a preimage in $\Hom_{\mathbf{Mon}^\infty}(\pmon{P},\pmon{Q})$ for a choice $\overline h\in \Hom(\sigmabar_Q,\sigmabar_P)$. By definition, the map $\overline h$ factors uniquely as
\[
\begin{tikzcd}
\sigmabar_Q\arrow[rr,"\overline h"]\arrow[dr,"h"']   &   &   \sigmabar_P \\
                    &\econe{\sigmabar_P/\tau} \arrow[ur,"i_\tau"'] & 
\end{tikzcd}
\]
for some face $\tau \preceq \sigma_P$, that is, $\overline h=(\tau, h:\sigma_Q\longrightarrow\sigma_P/\tau)$. By Lemma~\ref{imageofreese}, we know that $\iota_\tau$ is the image through $\Sigma$ of the Rees quotient morphism $r:\pmon{P}\longrightarrow \pmon{P}/\frakp_\tau$, thus we may write this diagram as
\[
\begin{tikzcd}
\Hom_{\mathbf{Mon}^\infty}(\pmon{Q},\R^\infty_{\geq0})\arrow[rr,"\overline h"]\arrow[dr,"h"']   &   &   \Hom_{\mathbf{Mon}^\infty}(\pmon{P},\R^\infty_{\geq0}) \\
                    &\Hom_{\mathbf{Mon}^\infty}(\pmon{P}/\frakp_\tau,\R^\infty_{\geq0}) \arrow[ur,"i_\tau"']. & 
\end{tikzcd}
\]
The morphism $h$ in this factorization is toric by definition and thus induces a morphism $h^\vee:\pmon{P}/\frakp_\tau\longrightarrow \pmon{Q}$. Our desired preimage is the composition

\begin{equation*}\label{dualtoh}
\begin{tikzcd}
\pmon{P}\arrow[rr]\arrow[dr,"r"']   &   &   \pmon{Q}\\
                    & \pmon{P}/\frakp_\tau\arrow[ur,"h^\vee"']. & 
\end{tikzcd}
\end{equation*}

To show faithful, let $f,f'\in \Hom(\pmon{P},\pmon{Q})$ with $\overline S(f)=\overline S(f')$. They both admit unique factorizations
\[
\begin{tikzcd}
\pmon{P}\arrow[rr,"f"] \arrow[dr,"r"']  &   &   \pmon{Q}\\
                    & \pmon{P}/\frakp_{\tau} \arrow[ur,"f_{\tau}"'] & 
\end{tikzcd}
\qquad \qquad
\begin{tikzcd}
\pmon{P}\arrow[rr,"f'"] \arrow[dr,"r'"']  &   &   \pmon{Q}\\
                    & \pmon{P}/\frakp_{\tau'} \arrow[ur,"f'_{\tau'}"'] & 
\end{tikzcd}
\]
for some faces $\tau$ and $\tau'$ of $\sigma_P$. Since $\overline S(f)=\overline S(f')$, the images of these factorizations through $\overline S$ both result in the same diagram
\[
\begin{tikzcd}
\sigmabar_Q\arrow[rr,"\overline S(f)=\overline S(f')"]\arrow[dr,"\overline S(r)=\overline S(r')"']   &   &   \sigmabar_P \\
                    &\econe{\sigma_P/\tau}=\econe{\sigma_P/\tau'}\arrow[ur,"i_\tau=i_{\tau'}"'] & 
\end{tikzcd}
\]
Thus $\tau=\tau'$, which forces $f_\tau=f'_{\tau'}$. Since this functor is already faithful for toric morphisms, we have $f=f'$. 
\end{proof}

\section{Pointed logarithmic structures}\label{section_pointedlogstr}

Let $\underline{X}$ be a scheme. Following K. Kato (see \cite[(1.1), (1.2)]{Kato_logstr}), a \emph{pre-logarithmic structure on $\underline{X}$} is a pair $(M_X, \alpha_X)$ consisting of a sheaf of monoids $M_X$ on the \'etale site of $\underline{X}_{et}$ together with a morphism $\alpha_X\mathrel{\mathop:}M_X \longrightarrow\calO_X$. A pre-logarithmic structure $(M_X,\alpha_X)$ on $\underline{X}$ is said to be a \emph{logarithmic structure}, if $\alpha_X$ induces a natural isomorphism 
\begin{equation*}
\alpha^{-1} \calO_X^\ast\xrightarrow{\sim}\calO_X^\ast \ .
\end{equation*}
In the following we are going to refer to the triple $(\underline{X},M_X,\alpha_X)$ simply as a \emph{logarithmic scheme} and denote it by $X$. 

A \emph{sheaf of pointed monoids} is a sheaf of monoids whose values lie in the category $\mathbf{Mon}^\infty$. Note that the structure sheaf $\calO_X$ of $X$ is a sheaf of pointed monoids on $X$ with absorbing element $0$. 

\begin{definition} 
A logarithmic structure $(M_X,\alpha_X)$ on $\underline{X}$ is a \emph{pointed logarithmic structure} if $M_X$ is a sheaf of pointed monoids and $\alpha_X$ is a morphism of sheaves of pointed monoids. 
\end{definition}

A triple $(\underline{X},M_X,\alpha_X)$ consisting of a scheme $\underline{X}$ together with a pointed logarithmic structure $(M_X,\alpha_X)$ will be referred to as a {pointed logarithmic scheme}. Logarithmic schemes form a category $\mathbf{LSch}$ (see \cite{Kato_logstr}). A \emph{morphism $f\mathrel{\mathop:}X\longrightarrow Y$ of logarithmic schemes} consists of a morphism $\underline{f}\mathrel{\mathop:}\underline{X}\longrightarrow \underline{Y}$ of the underlying schemes together with a morphism  $f^\flat\mathrel{\mathop:}\underline{f}^\ast M_Y\longrightarrow M_X$ of monoid sheaves on $\underline{X}$ that makes the natural diagram 
\begin{equation*}\begin{CD}
f^\ast M_Y@>f^\flat>>M_X\\
@Vf^\ast\alpha'VV @VV\alpha V\\
f^\ast\calO_{Y}@>f^\sharp>>\calO_X\\
\end{CD}\end{equation*}
commute. If $\underline{f}=\id_{\underline{X}}$ we refer to $f^\flat$ as a \emph{morphism of logarithmic structures} on $\underline{X}$.

\begin{definition}
A morphism $f\mathrel{\mathop:}X\longrightarrow Y$ of pointed logarithmic schemes is a morphism of logarithmic schemes such that $f^\flat\mathrel{\mathop:}f^\ast M_Y\longrightarrow M_X$ is a morphism of sheaves of pointed monoids. 
\end{definition}

The category of pointed logarithmic schemes is a faithful subcategory of the category of logarithmic schemes and will be denoted by $\mathbf{LSch}^\infty$.

\subsection{Pointification}\label{section:pointification}
There is a natural \emph{pointification functor}
\begin{equation*}
[.]^{\infty}\mathrel{\mathop:} \mathbf{LSch}\longrightarrow \mathbf{LSch}^\infty
\end{equation*}
that sends a logarithmic scheme $X$ to a pointed logarithmic scheme $X^\infty$ by applying the functor $[.]^\infty$ from Section \ref{section_pointedmonoids} on local sections; i.e. for an \'etale open $U$ on $X$ we have $M^\infty(U)=M(U)^\infty=M(U)\sqcup\{\infty\}$ and $\alpha^\infty$ is given by sending $\infty$ to $0\in\calO_X$. To avoid confusion, hereafter we restrict the use of this notation to identify pointed logarithmic schemes or stacks. We will note when an underlying monoid is pointed if it is not clear from context.

We refer to a morphism $f\mathrel{\mathop:}X^\infty\longrightarrow Y^\infty$ of pointed logarithmic schemes as a \emph{sub-logarithmic morphism}. If $f=g^\infty$ for a morphism $g\mathrel{\mathop:}X\longrightarrow Y$ of the underlying unpointed logarithmic schemes, i.e. if $f^\flat(m)=\infty_Y\in M_Y$ if and only if $m=\infty_X\in M_X$, we say that $f$ is a \emph{purely logarithmic morphism}.

A pointed logarithmic scheme is said to be \emph{coherent} (or \emph{fine}, \emph{separated}) if it is of the form $X^\infty$ for an (unpointed) logarithmic scheme $X$ that is \emph{coherent} (or \emph{fine}, \emph{separated}, respectively). Fix an algebraically closed base field $k$. In the following the term \emph{pointed logarithmic scheme} will refer to a fine and saturated pointed logarithmic scheme that is locally of finite type over a base field $k$.

\subsection{Geometry of sub-logarithmic morphisms}
We say that a morphism $f\mathrel{\mathop:}X\longrightarrow Y$ of pointed logarithmic schemes is a \emph{closed logarithmic immersion} if the underlying morphism $\underline{f}\mathrel{\mathop:}\underline{X}\longrightarrow \underline{Y}$ is a closed immersion and $f^\flat$ is surjective. 

\begin{proposition}\label{prop_Reesquotientlogschemes}
Let $Y$ be a pointed logarithmic scheme and $I\subseteq M_Y$ be an ideal sheaf. Then there is a unique logarithmic scheme $\mathcal{V}_Y(I)$ together with a closed logarithmic immersion $i_I\mathrel{\mathop:}\mathcal{V}_Y(I)\hookrightarrow Y$ such that, every sub-logarithmic morphism $f\mathrel{\mathop:}X\longrightarrow Y$ with $f^\flat(I)=\infty_X$ uniquely factors as 
\begin{equation*}\begin{CD}
X@>>> \mathcal{V}_Y(I)@>i_I>>Y
\end{CD}\end{equation*}
in $\mathbf{LSch}^\infty$.
\end{proposition}

In analogy with Section \ref{section_pointedmonoids} above we refer to the closed logarithmic subscheme $\mathcal{V}_Y(I)$ as the \emph{Rees quotient subscheme} of $Y$ by the ideal sheaf $I$. 

\begin{proof}[Proof of Proposition \ref{prop_Reesquotientlogschemes}]
Denote by $J$ the ideal sheaf in $\calO_Y$ that is generated by the non-unit elements in $I$. The underlying scheme of $\mathcal{V}_Y(I)$ is given as $\Spec \calO_Y/J$ or, in other words, as the closed subscheme of $Y$ defined by $J$. There is a unique factorization
\begin{equation*}\begin{tikzcd}
\underline{X} \arrow[r]&\underline{\mathcal{V}_Y(I)}\arrow[r] & \underline{Y}
\end{tikzcd}\end{equation*}
on the level of the underlying schemes. The closed subscheme is endowed with the pointed logarithmic structure $M/I$ that is given by the Rees quotients $M(U)/I(U)$ on \'etale opens $U$ of $Y$. The universal property of the Rees quotient on local sections then yields the existence and uniqueness of the factorization
\begin{equation*}\begin{tikzcd}
X \arrow[r]&\mathcal{V}_Y(I)\arrow[r] & Y
\end{tikzcd}\end{equation*}
in $\mathbf{LSch}^\infty$. 
\end{proof}

\begin{corollary}\label{cor_factoringlogschemes}
Let $f\mathrel{\mathop:}X\longrightarrow Y$ be a sub-logarithmic morphism. There is a unique factorization
\begin{equation*}\begin{CD}
X @>\tilde{f}>> \mathcal{V}_Y(I) @>i_I>> Y
\end{CD}\end{equation*}
where $\tilde{f}\mathrel{\mathop:}X\longrightarrow \mathcal{V}_Y(I)$ is a purely logarithmic morphism and $i_I\mathrel{\mathop:} \mathcal{V}_Y(I) \hookrightarrow Y$ is the Rees quotient subscheme of $Y$ by the ideal sheaf $I=(f^\flat)^{-1}(\infty_X)\subseteq M_X$. 
\end{corollary} 

\begin{proof} 
The preimage $I=(f^\flat)^{-1}(\infty_Y)$ is an ideal sheaf in $M_X$, thus the factorization coming from Proposition \ref{prop_Reesquotientlogschemes} yields the claim. Notice that by construction the logarithmic structure on $\mathcal{V}_Y(I)$ is given by identifying all elements in $(f^\flat)^{-1}(\infty_Y)$ with $\infty_{\mathcal{V}_Y(I)}$ and therefore the morphism $\tilde{f}\colon X\rightarrow \mathcal{V}_Y(I)$ is purely logarithmic. 
\end{proof}

\subsection{Fiber products} We now analyze fiber products in the category of fine and saturated pointed logarithmic schemes. 

\begin{proposition}\label{prop_fiberproductspointedlog}
The category of fine and saturated pointed logarithmic schemes admits a fiber product.
\end{proposition}
\begin{proof} In the case that we have morphism $X\rightarrow Z$ and $Y\rightarrow Z$ that are purely logarithmic (as defined in the previous section), then the fiber product is given by  $X^\infty\times_{Z^\infty}Y^\infty = (X\times_ZY)^\infty$ where $X\times_ZY$ denotes the fiber product in the category of fine and saturated logarithmic schemes. The concern is when $f:X\rightarrow Z$ and $g:Y\rightarrow Z$ are sub-logarithmic but not necessarily purely logarithmic. In this case, Corollary~\ref{cor_factoringlogschemes} furnishes unique factorizations.
\begin{equation*}\begin{CD}
X @>\tilde{f}>> \mathcal{V}_Z(I_f) @>i_{I_f}>> Z
\end{CD}\end{equation*}
and 
\begin{equation*}\begin{CD}
Y @>\tilde{g}>> \mathcal{V}_Z(I_g) @>i_{I_g}>> Z.
\end{CD}\end{equation*}
Denote by $(I_f\cup I_g)\subset M_Z$ the ideal sheaf on $Z$ generated by elements of $I_f$ and $I_g$ respectively. Denote by $J_f$ and $J_g$ the corresponding ideals in $\calO_Z$ generated respectively by their non-unit elements, and $(J_f+J_g)$ the sum of ideals. Note that $(J_f+J_g)$ is precisely the ideal sheaf generated in $\calO_Z$ by $(I_f\cup I_g)$.  Then the pushouts
\begin{equation*}
\begin{tikzcd}
M_Z\arrow[r]\arrow[d] & M_Z/I_g\arrow[d]\\
M_Z/I_f\arrow[r]& M_Z/(I_f\cup I_g)
\end{tikzcd}
\hspace{2cm}
\begin{tikzcd}
\calO_Z\arrow[r]\arrow[d] & \calO_Z/J_g\arrow[d]\\
\calO_Z/J_f\arrow[r]& \calO_Z/(J_f+J_g)
\end{tikzcd}
\end{equation*}
determine a Cartesian square
\begin{equation*}
\begin{tikzcd}
\mathcal{V}_Z(I_f\cup I_g)\arrow[r]\arrow[d]   & \mathcal{V}_Z(I_g) \arrow[d,"i_{I_g}"] \\
\mathcal{V}_Z(I_f)\arrow[r,"i_{I_f}"']&Z.
\end{tikzcd}
\end{equation*}
in the category of fine and saturated pointed logarithmic schemes. Taking the smallest ideal sheafs generated by $f^\flat(I_g)$ and $g^\flat(I_f)$ (as in Lemma~\ref{pushoutlemma}) in $M_X$ and $M_Y$ respectively, we can extend this to a diagram of Cartesian squares
\begin{equation*}
\begin{tikzcd}
& \mathcal{V}_Y(g^\flat(I_f)+M_Y)\arrow[r]\arrow[d]   & Y\arrow[d,"\tilde{g}"] \\
\mathcal{V}_X(f^\flat(I_g)+M_X)\arrow[r]\arrow[d]  & \mathcal{V}_Z(I_f\cup I_g)\arrow[r]\arrow[d]   & \mathcal{V}_Z(I_g) \arrow[d,"i_{I_g}"] \\
X\arrow[r,"\tilde{f}"']&\mathcal{V}_Z(I_f)\arrow[r,"i_{I_f}"']&Z.\\
\end{tikzcd}
\end{equation*}
Since both $\tilde{f}$ and $\tilde{g}$ are purely logarithmic, the top-left two arrows in this diagram are as well. Thus we may construct our desired fiber product $X\times_Z Y$ by completing this to a large Cartesian square consisting of four Cartesian squares.
\end{proof}

\begin{example}\label{curveexample} Consider the family of affine curves $\pi:C=\Spec{\k[x,y,t]/(xy-t)}\longrightarrow \Spec{\k[t]}$ with special fiber at $t=0$ the nodal singularity given by the union of two axes. The logarithmic structures determined by charts $\N^2\rightarrow \k[x,y,t]/(xy-t)$ and $\N\rightarrow \Spec{\k[t]}$ produce a family of logarithmically smooth curves via the diagonal $\Delta:\N\rightarrow \N^2$:
\[
\begin{tikzcd}
\N^2\arrow[r,"x^ny^m"]   & \k[x,y,t]/(xy-t)  \\
\N\arrow[u,"\Delta"]\arrow[r,"t^n"']&\k[t]\arrow[u].
\end{tikzcd}
\]
Restricting to the special fiber results in a node over the standard logarithmic point \[\pi_0:C_0 \longrightarrow (\Spec{\k},\k^\ast\oplus \N)\] with the characteristic monoid $\overline{M}_{C_0}$ taking the form
\begin{equation*}
\N a\oplus\N b\oplus \N\delta/ (a+b=\delta)
\end{equation*}
where $a,b$ are generators for $\N^2$ and $\delta =\pi_0^\flat(1)$. While pointing the logarithmic structures above does not immediately alter the situation as these morphisms are all logarithmic, it does allows us the option to pull back along the sub-logarithmic morphism \[(\Spec{\k},\k^\ast\oplus \{0,\infty\})\rightarrow (\Spec{\k},\k^\ast\oplus \N^\infty)\] which, at the level of monoids, is determined by the Rees quotient $\N^\infty\longrightarrow \{0,\infty\}$ sending $0\mapsto 0$ and $n\mapsto \infty$ for $n>0$. This results is a pointed affine logarithmic curve over the pointed logarithmic point $(\Spec{\k},\k^\ast\oplus \{0,\infty\})$ carying what is, in a sense, an intrinsic pointed logarithmic structure, i.e.\, the analogous Rees quotient caries through upstairs resulting in a characteristic monoid on the pullback of $C_0$ near the node where $\delta=\infty_{C_0}$ in $M_{C_0}$.
\end{example}

\subsection{Logarithmic stratifications}\label{logstrat}
Let $X$ be a fine and saturated logarithmic scheme that is logarithmically smooth over $k$. Suppose moreover for simplicity that $X$ has no self-intersection, i.e. that $M_X$ is defined in the Zariski topology on $X$. In this case there is natural stratification of $X$ by locally closed subsets that are in a one-to-one correspondence with the cones in the cone complex associated to $X$ (see \cite[Corollary 3.5]{Ulirsch_tropcomplogreg}).  

Let us quickly recall the inductive construction of this stratification: Let $X_0$ be the locus
\begin{equation*}
X_0=\big\{x\in X\big\vert M_{X,x}\simeq\calO_{X,x}^\ast\big\} \ .
\end{equation*}
Then $X_0$ is smooth and the connected components of $X_0$ are the open strata of $X$ of codimension zero. For $n\geq 1$ denote by $X_n$ the locally closed subset of regular points of $X-(X_0\cap \cdots \cap X_{n-1})$. The irreducible components of $X_n$ form the strata of $X$ of codimension $n$. 

Let $E$ be a locally closed stratum of $X$ and denote by $\overline{E}$ its closure in $X$. Let $I_{\overline{E}}$ be the preimage in $M_X$ of the vanishing ideal of $\overline{E}$ in $\calO_X$. Then $I_{\overline{E}}$ is an ideal sheaf in $M_X$ and $\overline{E}$ is naturally isomorphic to the scheme underlying the Rees quotient subscheme $\mathcal{V}_X(I_E)$. This construction naturally endows the closure $\overline{E}$ of a stratum with a logarithmic structure that makes the inclusion $\overline{E}\hookrightarrow X$ into a morphism of pointed logarithmic schemes. 

In other words, we promote the natural stratification of $\underline{X}$ in the category of scheme to a stratification of $X$ in the category of pointed logarithmic schemes.

\begin{remark}
Let $X=X(\Delta)$ be a toric variety associated to a rational polyhedral fan $\Delta$ in $N_\R=N\otimes \R$, where $N$ is the cocharacter lattice of the big torus $T$ of $X$. There is a natural stratification of $X$ by the $T$-orbits of $X$ that are in a natural one-to-one correspondence with the cones in $\Delta$. Write $O(\sigma)$ for the locally closed $T$-orbit of $X$ associated to a cone $\sigma$ in $\Delta$. Then the natural immersion $\overline{O}(\sigma)\hookrightarrow X$ of the closure $O(\sigma)$ is not toric, but rather \emph{subtoric} (i.e.\ a composition of a toric morphism, here the identity, with the closed immersion of a closed torus orbit).
\end{remark}

\section{Clutching and gluing in logarithmic geometry}\label{section_logarithmicclutching&gluing}

Let $S$ be a logarithmic scheme. In \cite{Kato_logsmoothcurves} F. Kato has defined a logarithmic curve over $S$ as a logarithmically smooth integral morphism $X\rightarrow S$ (later also assumed to be proper) such that each geometric fiber is a reduced and connected curve with at worst nodal singularities (i.e. pre-stable). In particular, in \cite{Kato_logsmoothcurves}*{Section 1} Kato gives a combinatorial characterization of the \'etale local structure of such curves. We directly generalize F. Kato's characterization to our setting (using the notation from \cite[Section 7.2]{CavalieriChanUlirschWise_tropstack}). 

\begin{definition}\label{def_logcurve}
Let $S$ be a pointed logarithmic scheme. A \emph{(pointed) logarithmic curve} is a sublogarithmic morphism $\pi\colon X\rightarrow S$ whose underlying morphism of schemes is proper and flat with each fiber a reduced and connected curve with at worst nodal singularities (i.e. pre-stable), such that every geometric point $x$ of $X$ has an \'etale neighborhood $U$ that admits a strict \'etale morphism to a logarithmic scheme $V$ over $S$ that is of one of the following three types:
\begin{enumerate}[(i)]
\item $V=\Spec \calO_S[u]$, with $M_V=\pi^\ast M_S$;
\item $V=\Spec \calO_S[u]$, with $M_V=\pi^\ast M_S\oplus \N v$ with $\pi^\flat(v)=u$; or
\item $V=\Spec \calO_S[x,y]/(xy-t)$ for some $t\in \calO_S$, and 
\begin{equation*}
M_V=\pi^\ast M_S\oplus \calO_S^\ast\alpha \oplus \calO_S^\ast\beta / (\alpha +\beta =\delta)
\end{equation*}
for some $\delta\in M_S$ and $\alpha \mapsto x$, $\beta\mapsto y$, and $\delta\mapsto t$ through their respective structure morphisms.
\end{enumerate}
\end{definition}

\noindent The main difference from the situation in \cite{Kato_logsmoothcurves} is that in the situation of Part (iii) above we now allow $\delta= \infty$. We denote by $\calMbar_{g,n}^{log}$ the category of stable logarithmic curves.

\begin{remark}
Rather than developing here a fully fledged theory of logarithmically smooth morphisms in the category of pointed logarithmic schemes, we have restricted our present study to the appropriate generalization of logarithmic curves (guided by their \'etale local structure in the unpointed setting) given above in Definition~\ref{def_logcurve}. A complete treatment of logarithmically smooth morphisms of pointed logarithmic schemes is warranted as we believe it will offer a slightly richer theory. This remains for future work. As a quick aside in this direction, note that logarithmic smoothness is a property of morphisms that is preserved under base change along logarithmic maps. This should still be the case for base change of smooth purely logarithmic morphisms along sub-logarithmic morphisms as well. Example~\ref{curveexample} carries out such a base change of a logarithmically smooth family of affine curves.
\end{remark}

As in the unpointed case, $\calMbar_{g,n}^{log}$ forms a category fibered in groupoids over $\mathbf{LSch}^\infty$ whose fiber over a pointed logarithmic scheme is the groupoid of stable logarithmic curves over $S$ of genus $g$ with $n$ marked sections: the functor $\calMbar_{g,n}^{log}\longrightarrow \mathbf{LSch}^\infty$ assigns to a family its base pointed logarithmic scheme. Morphisms in this category are the usual commutative diagrams. As the above definition is stable under strict base change, using the strict \'etale topology $\tau_{str.et.}$ it is straightforward to check that $\calMbar_{g,n}^{log}$ forms a stack over $(\mathbf{LSch}^\infty, \tau_{str.et.})$. 

While there are new objects in $\calMbar_{g,n}^{log}$ that do not appear in the unpointed case as a smoothing parameter $\delta$ may now take the value $\infty$ in $M_X$, the unique factorization of sub-logarithmic morphisms keeps this well controlled and allows us to follow Kato's original constructions closely. Indeed, the following Proposition \ref{prop_Mbarlog=Mbar} is an immediate generalization of \cite{Kato_logsmoothcurves}*{Theorem 4.5}. In \cite{Kato_logsmoothcurves}, Kato proves that the category of unpointed stable logarithmic curves is representable by $\calMbar_{g,n}$ with the induced logarithmic structure from its boundary by first identifying the subcategory of appropriate ``basic" or minimal objects. This subcategory consists of canonical logarithmically smooth curves with logarithmic structure on the base $S$ induced as the \'etale local lift of a morphism of monoids $\N^r\rightarrow \calO_S$. In the present setting, we identify the minimal objects of $\calMbar_{g,n}^{log}$ using pointifications of the the same canonical logarithmic curves, their pointed logarithmic structures induced now by $(\N^r)^\infty \rightarrow \calO_S$. This furnishes a subcategory of $\calMbar_{g,n}^{log}$ that is again representable by $\calMbar_{g,n}$, but now with the pointification of the induced logarithmic structure from its boundary. 

\begin{proposition}\label{prop_Mbarlog=Mbar}
The stack $\calMbar_{g,n}^{log}$ over $\big(\mathbf{LSch^\infty}, \tau_{str.et.}\big)$ is representable by the algebraic stack $\calMbar_{g,n}$ with pointed logarithmic structure $M:=M_{\calMbar_{g,n}}$ given by the pointification of the standard divisorial logarithmic structure induced from its boundary. 
\end{proposition}

\begin{proof} 
As absorbing elements must restrict to absorbing elements, we note that the restriction of $M$ to any family $\underline{X}\longrightarrow \underline{S}$ in $\calMbar_{g,n}$ results in the pointification of the canonical stable log.\ curve in the sense of \cite{Kato_logsmoothcurves}, and all such pointified canonical stable log.\ curves arise this way. Recall, this is the logarithmic curve $\underline{X}\longrightarrow \underline{S}$ with $M_S$ given by the \'etale local lift of a morphism of monoids $(\N^r)^\infty \rightarrow \calO_S$ where $r$ is the maximum number of nodes of any fiber, and each choice $\delta$ the corresponding generator of $\N^r$ for each node.

To extend the proof of \cite{Kato_logsmoothcurves}*{Theorem 4.5} to the pointed setting, it is enough to prove that pointified canonical stable log.\ curves satisfy the pull-back criterion for minimality (see \cite{Kato_logsmoothcurves}*{Proposition~2.3} where the word ``basic" is used, or \cite{Gillam-Minimality}) with respect to morphisms in $\mathbf{LSch}^\infty$. To this end, let $f':X'\longrightarrow S'$ be a family in $\calMbar_{g,n}^{log}$ and $f:X\longrightarrow S$ a pointified canonical stable log. curve with
\[
\begin{tikzcd}
\underline{X'}\arrow[r]\arrow[d,"\underline{f'}"'] & \underline{X}\arrow[d,,"\underline{f}"]\\
\underline{S'}\arrow[r]&\underline{S}
\end{tikzcd}
\]
a cartesian diagram of schemes. We must show there exist unique $a:S'\longrightarrow S$ and $b:X'\longrightarrow X$ lifting this to a cartesian diagram
\[
\begin{tikzcd}
X'\arrow[r,"b"]\arrow[d,"f'"'] & X\arrow[d,"f"]\\
S'\arrow[r,"a"']&S
\end{tikzcd}
\]
in the category of pointed logarithmic schemes. Following \cite{Kato_logsmoothcurves}*{Proposition~2.3}, these maps are uniquely determined \'etale locally by the image of the generators $\{e_i\}$ of $\N^r$, each corresponding to the $i$-th node with image giving the value $\delta_i$ in $M_S'$.

Indeed, \'etale locally near a node $X'$, the structure of a logarithmic curve determines the value $\delta_i$ in $M_S'$.  The logarithmic map $a$ is determined by sending $e_i$ to $\delta_i$ for $i=1\ldots r$. By Corollary~\ref{cor_factoringlogschemes}, $a$ factors uniquely as 
\[
\begin{tikzcd}
S'\arrow[r,"\tilde a"]& \mathcal{V}_S(I) \arrow[r,"i_I"]& S
\end{tikzcd}
\]
where $\tilde a$ is purely logarithmic and $i_I$ is a closed logarithmic immersion and \[I=\langle e_i|\delta_i:=a^\flat(e_i)=\infty_{S'} \rangle \simeq \N^{r-s}\] where $s$ is the maximum number of $\delta_i$ taking the value $\infty_{S'}$. These maps are determined at the level of characteristics by the composition of a Rees quotient and a toric morphism:
\[\pmon{\left(\N^r\right)}\longrightarrow\pmon{(\N^r)}/I\longrightarrow \overline{M_S'}.\]
The gluing argument of \cite{Kato_logsmoothcurves}*{Proposition~2.3} and Lemma~\ref{pushoutlemma} uniquely determine a commutative diagram
\begin{equation}\label{curvefactorization}
\begin{tikzcd}
X'\arrow[r,"\tilde b"]\arrow[d,"f'"'] & \mathcal{V}_X(I')\arrow[d]\arrow[r,"i_{I'}"]& X\arrow[d,"f"]\\
S'\arrow[r,"\tilde a"']&\mathcal{V}_S(I)\arrow[r,"i_I"]&S
\end{tikzcd}
\end{equation}
with $I'=(f^\flat(I)+M_X')$ and the right square cartesian. The left square is cartesian by \cite{Kato_logsmoothcurves}*{Corollary~1.16}, completing the proof.
\end{proof}

\begin{remark}
The commutative Diagram~\ref{curvefactorization} utilizes the unique factorization of a sub-logarithmic morphism along the base to factorize the morphism of curves in a similar manner. The logarithmic stratum of $X'\longrightarrow S'$ into which the family $X\longrightarrow S$ maps is identified combinatorially by the smoothing parameters $\delta_i$ taking the value $\infty$. 
\end{remark}

\begin{theorem}\label{thm_logclutch&glue}
The classical clutching and gluing maps induce natural sub-logarithmic clutching maps 
\begin{equation*}\begin{split}
\calMbar_{g,n+1}^{log}\times\calMbar_{g',n'+1}^{log}&\longrightarrow \calMbar_{g+g',n+n'}^{log}\\
\big([C,s_1,\ldots, s_n, \star],[C',s_1',\ldots, s'_{n'}, \bullet]\big)&\longmapsto \big[C\sqcup_{\star\sim\bullet} C', s_1,\ldots, s_n, s_1',\ldots, s_{n'}'\big] \\
\end{split}\end{equation*}
and (self-)gluing maps
\begin{equation*}\begin{split}
\calMbar_{g-1,n+2}^{log}&\longrightarrow \calMbar_{g,n}^{log} \\
\big[ C,s_1,\ldots, s_n, \star, \bullet\big] &\longmapsto \big[C/_{\star\sim\bullet},s_1,\ldots, s_n\big] \ .
\end{split}\end{equation*}
\'Etale locally around the new node $\star\sim\bullet$ the logarithmic structure (in the notation of Definition \ref{def_logcurve} (iii) above) is given by 
\begin{equation*}
M_V=\pi^\ast M_S\oplus \calO_S^\ast\alpha \oplus \calO_S^\ast\beta / (\alpha +\beta =\delta)
\end{equation*}
where $\delta=\infty_S\in M_S$. 
\end{theorem}

\begin{proof}
The two formulas above define a functor between the fibered categories over $\mathbf{LSch}^\infty$ whose restriction to the category of small pointed logarithmic schemes (in the sense of \cite{Kato_logsmoothcurves}) are exactly the classical clutching and gluing morphisms for the algebraic moduli stacks $\calMbar_{g,n}$. This observation, together with Proposition \ref{prop_Mbarlog=Mbar}, implies the claim. 
\end{proof}

\begin{remark}
Let $S=\big(\Spec k, k^\ast \oplus P^\circ\big)$ be a standard logarithmic point, where $P^\circ$ is a sharp (unpointed) monoid. In \cite{FosterRanganathanTalpoUlirsch_logPic}*{Theorem 4.4} we have seen that the groupoid $\calM_{g,n}^{log}(S)$ is equivalent to the groupoid of metrized curve complexes (in the sense of \cite{AminiBaker_metrizedcurvecomplexes}) with edge lengths in $P^\circ$. If we take $S^\infty=\big(\Spec k, k^\ast\oplus P^\infty\big)$, this equivalence extends to an equivalence between $\calMbar_{g,n}^{log}(S^\infty)$ and the groupoid of metrized curve complexes with edge lengths in $P^\infty$. From this point of view, the clutching map is given by connecting the two marked points $\star$ and $\bullet$ by an edge of length $\infty$. 
\end{remark}

\section{Clutching and gluing in tropical geometry}\label{section_tropicalclutching&gluing}

In order to find tropical analogues for the \emph{clutching and gluing maps}, as originally introduced in \cite{Knudsen_projectivityII}, it has been realized in \cite[Section 8]{ACP} that, instead of working with morphisms of generalized cone complexes, one has to consider morphisms between their canonical compactifications. In this section we apply the same principle to the tropical moduli stacks introduced in \cite{CavalieriChanUlirschWise_tropstack} and construct tropical clutching and gluing maps that naturally commute with tropicalization. 

\begin{remark}
We recall from Section~\ref{foundations} that while the data of a strictly convex rational polyhedral cone $(\sigma,N)$ includes the choice a lattice $N$ giving rational structure, we will often abuse notation below and simply write $\sigma$ for the cone $(\sigma,N)$ unless otherwise necessary to avoid confusion.
\end{remark}

\subsection{Extended cone stacks}
A \emph{(rational polyhedral) cone complex} $\Sigma$ is a topological space $\vert \Sigma\vert$ that arises as a colimit of a diagram consisting of sharp cones and proper face morphisms such that the induced maps $\sigma\rightarrow\vert\Sigma\vert$ (called the faces of $\Sigma$) are injective. We refer to the reader to \cite{CavalieriChanUlirschWise_tropstack}*{Definition 2.1} for a more axiomatic definition of this notion. The category of cone complexes will be denoted by $\mathbf{RPCC}$; its morphisms are continuous maps $\Sigma\rightarrow\Sigma'$ that restrict to morphisms of cones on the faces of $\Sigma$. 

The category $\mathbf{RPCC}$ naturally carries a Grothendieck topology, the so-called \emph{face topology}, given by face embeddings $\sigma\hookrightarrow\Sigma$. A morphism $\Sigma\rightarrow\Sigma'$ in $\mathbf{RPCC}$ is said to be \emph{strict} if its restriction to a face of $\Sigma$ induces an isomorphism onto a face of $\Sigma'$. Denote by $\PP_{strict}$ the class of strict morphisms. The tuple $\big(\mathbf{RPCC},\tau_{face},\PP_{strict}\big)$ therefore defines a geometric context in the sense of \cite{CavalieriChanUlirschWise_tropstack}*{Section 1} and we may define cone spaces and cone stacks as geometric spaces and geometric stacks in this context respectively. Recall from \cite{CavalieriChanUlirschWise_tropstack}*{Section 1} that a geometric context is a tuple consisting of a category, a Grothendieck topology, and a class of morphisms $\PP$ fulfilling, in a sense, the minimal set of axioms necessary to allow for the construction of geometric stacks. We refer the reader to \cite{CavalieriChanUlirschWise_tropstack} for a thorough description of how this works in the case of cone stacks that we employ below.

Given a cone complex $\Sigma$, we may define its canonical extension $\Sigmabar$ by replacing every cone $\sigma$ in $\Sigma$ by its canonical extension $\sigmabar$ and gluing along the face incidences of $\Sigma$. Following \cite[Section 2.4]{ACP} we call complexes formed in this way \emph{extended rational polyhedral cone complexes} or just \emph{extended cone complexes} for short. We refer to the $\sigmabar\in\Sigmabar$ as the \emph{extended cones} of $\Sigmabar$. A \emph{morphism of extended cone complexes} (again, following \cite[Section 2.4]{ACP}) is a continuous map $f\colon \Sigmabar\rightarrow \Sigmabar'$ such that for each extended cone $\sigmabar$ in $\Sigmabar$ there is an extended cone $\sigmabar'$ of $\Sigmabar'$ such that $f\vert_{\sigmabar}$ induces a morphism of extended cones $\sigmabar\rightarrow \sigmabar'$. We denote by $\mathbf{RPCC}^\infty$ the category of extended cone complexes. 

For a cone $\sigma\in\Sigma$, we denote by $\Star(\sigma,\Sigma)$ the cone complex formed by taking the collection of quotients $\delta/\sigma$ for all $\delta\in\Sigma$ containing $\sigma$ as a face:
\[\Star(\sigma,\Sigma)=\{\delta/\sigma | \sigma \preceq \delta \in \Sigma\}.\] Two cones $\delta/\sigma$ and $\delta'/\sigma$ meet in a face $\tau/\sigma$ precisely when $\sigma\preceq\tau\preceq\delta$ and $\sigma\preceq\tau\preceq\delta'$. 

The stratifications \[ \sigmabar=\bigsqcup_{\tau\preceq \sigma}\sigma/\tau\] given in Proposition~\ref{prop_extendedcones} glue among all the extended cones in $\Sigmabar$, resulting in an induced stratification
\[\Sigmabar=\bigsqcup_{\sigma\in\Sigma}\Star(\sigma,\Sigma)\]
by maximal extended cone complexes at infinity, each canonically isomorphic to $\Star(\sigma,\Sigma)$. We denote by
\[\overline{\Star}(\sigma,\Sigma)=\bigsqcup_{\sigma\preceq\delta}\Star(\delta,\Sigma)\]
 the canonical extension of such a maximal complex at infinity. The inclusion maps $\overline{\delta/\sigma}\hooklongrightarrow\overline{\delta}$ for cones $\delta$ containing $\sigma$ as a face glue together, inducing an inclusion
\[i_{(\sigma,\Sigma)}:\overline{\Star}(\sigma,\Sigma)\hooklongrightarrow \Sigmabar.\]

Recall that any morphism $\sigma\longrightarrow\sigma'$ of cones induces a canonical extension $\sigmabar\longrightarrow\sigmabar'$ by continuity (we called such maps of extended cones ``toric" in Section~\ref{section_extendedcones}). Thus every morphism $f:\Sigma\rightarrow\Sigma'$ of cone complexes induces a canonical extension to a morphism $\Sigmabar\rightarrow \Sigmabar'$ of extended cone complexes, so that the natural functor
\begin{equation*}\begin{split}
[.]^\infty\mathrel{\mathop:}\mathbf{RPCC}&\longrightarrow\mathbf{RPCC}^\infty\\
\Sigma&\longmapsto \Sigmabar
\end{split}\end{equation*}
is faithful and essentially surjective. In a slight abuse of notation we denote the induced morphism by the same name $f:\Sigmabar\longrightarrow\Sigmabar'$ and call it a \emph{toroidal} morphism of extended cone complexes. 

As in the case of extended cones described in Section~\ref{section_affine}, we note the following important feature of extended cone complexes: while the new extended category $\mathbf{RPCC}^\infty$ does not introduce new objects (since the compactifications of complexes are canonical), it does introduce new morphisms. Indeed, a morphism $f:\Sigmabar\rightarrow \Sigmabar'$ of extended cone complexes may take its image in some maximal extended cone complex at infinity $\overline{\Star}(\sigma',\Sigma')\subset \Sigmabar'$, avoiding the underlying complex $\Sigma'$ entirely. Toroidal morphisms are precisely the morphisms induced from underlying morphisms of complexes, i.e., the morphisms whose image meets the underlying complex $\Sigma'$. We called such morphisms of extended cones \emph{toric} in Section~\ref{section_affine}. A morphism of extended cone complexes $f:\Sigmabar\longrightarrow\Sigmabar'$ must factor uniquely as a toroidal morphism composed with the inclusion $i_{(\gamma',\Sigma')}:\overline{\Star}(\gamma',\Sigma')\hooklongrightarrow \Sigmabar'$ of the maximal extended cone complex at infinity where it takes its image. Both $\gamma'$ and the factorization are determined locally, i.e. cone by cone, by the unique factorization of morphisms of extended cones. 

\begin{proposition}\label{prop_factorization} Given a morphism $f\colon\Sigmabar\rightarrow \Sigmabar'$ of extended cone complexes, there is a unique factorization
\begin{equation*}
\begin{CD}
f:\Sigmabar @>\widetilde{f}>> \overline{\Star}(\gamma',\Sigma') @>i_{(\gamma',\Sigma')}>>\Sigmabar'
\end{CD}
\end{equation*}
where $\gamma'$ is a cone in $\Sigma'$ and $\widetilde{f}$ is a toroidal morphism of extended cone complexes. 
\end{proposition}

\begin{proof}
Restricting $f\colon\Sigmabar\rightarrow \Sigmabar'$ to any extended cone $\sigmabar\in\Sigmabar$ determines a morphism of extended cones $f|_{\sigmabar}:\sigmabar\longrightarrow \sigmabar'$ which uniquely factors as a toric morphism followed by an inclusion
\[
\begin{tikzcd}
\sigmabar\arrow[rr,"f|_{\sigmabar}"]\arrow[dr]   &   &   \sigmabar' \\
                    & \econe{\sigma'/\tau'} \arrow[ur,hook,"i_{\tau'}"']. & 
\end{tikzcd}
\]
These factorizations glue to form the desired factorization of $f$. Uniqueness is ensured by the requirement that $\tilde{f}$ be toroidal, expressed cone by cone by requiring that $\sigmabar\longrightarrow\econe{\sigma'/\tau'}$ above is toric.
\end{proof}

\begin{remark}\label{ref_unique}
Let $f:\Sigmabar\longrightarrow\Sigmabar'$ be a morphism of extended cones and $\gamma'\in\Sigma'$ the cone determined by Proposition~\ref{prop_factorization}.  Forgoing the hypothesis that $\tilde{f}$ is toroidal would allow for other factorizations 
\[
\begin{CD}
f:\Sigmabar @>\widetilde{f}>> \overline{\Star}(\delta',\Sigma') @>i_{(\delta',\Sigma')}>>\Sigmabar',
\end{CD}
\]
one for each face $\delta'\preceq\gamma'$. In particular there is always a trivial factorization
\[
\begin{CD}
f:\Sigmabar @>\widetilde{f}>> \overline{\Star}(0,\Sigma') @>i_{(0,\Sigma')}>>\Sigmabar'
\end{CD}
\]
where the second morphism is the identity.
\end{remark}

Proposition \ref{prop_factorization} allows us to apply the same argument as in Proposition \ref{prop_fiberproductspointedlog} to show that the category of extended cone complexes admits fiber products. The face maps $\sigmabar\hookrightarrow \Sigmabar$ for extended cones $\sigmabar$ of $\Sigmabar$ induce a Grothendieck topology $\tau_{face}$ on $\mathbf{RPCC}^{\infty}$. We say that morphism $f\colon\Sigmabar\rightarrow\Sigmabar'$ is \emph{strict}, if the restriction of $f$ to every extended cone $\sigmabar$ in $\Sigmabar$ induces an isomorphism onto an extended cone in $\Sigmabar'$. 

Denote by $\PP_{strict}$ the class of strict morphisms.  The tuple $(\mathbf{RPCC}^\infty, \tau_{strict}, \PP_{strict})$ defines a geometric context in the sense of \cite[Definition 1.1]{CavalieriChanUlirschWise_tropstack}. We may now define an \emph{extended cone stack} as a geometric stack in the context $\big(\mathbf{RPCC}^\infty, \tau_{face},\PP_{strict}\big)$.

\begin{definition}
An \emph{extended cone stack} is a stack $\calCbar$ over $(\mathbf{RPCC}^\infty, \tau_{face})$ that fulfills the following axioms:
\begin{enumerate}[(i)]
\item The diagonal morphism $\Delta\colon \calCbar\rightarrow\calCbar\times\calCbar$ is representable by a extended cone complexes. 
\item There is an extended cone complex $\Sigmabar$ as well as a (necessarily representable) morphism $\Sigmabar\rightarrow \calCbar$ that is strict and surjective. 
\end{enumerate}
\end{definition}

Given a strict and surjective groupoid object $(R\rightrightarrows U)$ in $\mathbf{RPCC}$, the induced groupoid $(R^\infty\rightrightarrows U^\infty)$ in $\mathbf{RPCC}^\infty$ is again strict and surjective. Thus, there is a natural functor 
\begin{equation*}
[.]^{\infty}\colon\mathbf{ConeStacks}\longrightarrow \mathbf{ConeStacks}^\infty
\end{equation*}
such that, whenever $[U/R]\simeq \calC$ is a groupoid presentation of a cone stack $\calC$, we have a natural equivalence $[U^\infty/R^\infty]$. 

\begin{remark}
In \cite[Section 2]{CavalieriChanUlirschWise_tropstack} we first define the notion of a cone space, an analogues of an algebraic space, and then require the diagonal of a cone stack to only be representable by cone spaces (and not by cone complexes). We refrain from including this extra complication here, but mention this technical fine print for the interested reader.  
\end{remark}

\subsection{The extended moduli stack $\calMbar_{g,n}^{trop}$} In \cite{CavalieriChanUlirschWise_tropstack}*{Section 3} moduli stacks $\calM_{g,n}^{trop}$ of tropical curves are defined as cone stacks whose fiber over a sharp cone $\sigma$ is the groupoid of stable tropical curves of genus $g$ with $n$ marked legs and edge lengths in the dual monoid $S_\sigma$. We are now going to extend this construction to tropical curves with edge lengths in pointed monoids. 

\begin{definition}
Let $\sigmabar$ be an extended cone. An \emph{extended tropical curve} $\Gamma$ over $\sigmabar$ is a finite  graph $G$ (possibly with legs) together with a non-negative vertex weight $h\colon V(G)\rightarrow \Z_{\geq 0}$ and an edge length function $d\colon E(G)\rightarrow S_\sigma^\infty-\{0\}$.
\end{definition}

Recall that the \emph{genus} of $\Gamma$ (and also of $G$) is defined to be the number 
\begin{equation*}
g(\Gamma)=g(G)=b_1(G)+\sum_{v\in V(G)}h(v)
\end{equation*} 
and that $\Gamma$ is said to be \emph{stable} if for all vertices $v\in V(G)$ the inequality $2h(v)-2+\vert v\vert>0$ holds. 

The argument in \cite{CavalieriChanUlirschWise_tropstack}*{Proposition 2.3} immediately shows that there is a unique stack $\calMbar_{g,n}^{trop}$ over the site $\big(\mathbf{RPCC}^\infty, \tau_{face}\big)$ (up to equivalence)  whose fiber over an extended cone $\sigmabar$ is the groupoid of stable \emph{extended tropical curves} $\Gamma$ over $\sigmabar$ of genus $g$ with $n$ marked legs. 

\begin{proposition}\label{prop_Mgn=Mbargn}
The moduli stack $\calMbar_{g,n}^{trop}$ is an extended cone stack.   
\end{proposition}

\begin{proof}
Fix a stable vertex-weighted finite graph $G=(V,E,L,h,m)$ of genus $g$ with $n$ marked points. The moduli functor $\Ubar_G$ whose fiber over an extended rational polyhedral cone $\sigmabar$ is the groupoid of pairs $(\Gamma, \phi)$ consisting of a tropical curves $\Gamma$ in $\calMbar_{g,n}^{trop}(\sigmabar)$ as well as a weighted edge contraction $\phi\colon G\rightarrow \GG(\Gamma)$, where $\GG(\Gamma)$ denotes the underlying weighted finite $n$-marked graph of $\Gamma$.

Lemma 3.4 in \cite{CavalieriChanUlirschWise_tropstack} immediately generalizes to this situation and so $\Ubar_G$ is representable by the extended rational polyhedral cone $\sigmabar_G=\Rbar_{\geq 0}^E$. The main point here is that, for every extended cone $\sigmabar$, there is a one-to-one correspondence between the two sets $\Ubar_G(\sigmabar)$ and 
\begin{equation*}
\Hom(\sigmabar, \sigmabar_G)=\Hom\big(\N^{E}, S_{\sigma}^\infty\big) \ .
\end{equation*}
Given a homomorphism $f\colon \N^{E}\rightarrow S_\sigma^\infty$, we endow $G$ with a generalized edge length given by 
\begin{equation*}
d_f(e)=f\big([e]\big)
\end{equation*}
for $e\in E$ and the associated tropical curve in $\Ubar_G(\sigmabar)$ is given by contracting all the edges $e$ for which $d_f(e)=0$. 

As in \cite[Lemma 3.5]{CavalieriChanUlirschWise_tropstack}, the natural morphism 
\begin{equation*}\begin{split}
\Ubar_G&\longrightarrow \calMbar_{g,n}^{trop}\\
(\Gamma/\sigmabar, \phi)& \longmapsto \Gamma/\sigmabar
\end{split}\end{equation*}
representable by extended cone complexes, strict, and quasi-compact. 

This implies that $\Ubar_{g,n}=\bigsqcup_{G}\Ubar_G$ is a cone complex and that the natural morphism 
\begin{equation*}\begin{split}
\Ubar_{g,n}&\longrightarrow \calMbar_{g,n}
\end{split}\end{equation*}
is representable, strict, and quasi-compact. Since it is clearly surjective, we have constructed a representable strict and surjective atlas of the stack $\calMbar_{g,n}^{trop}$. Finally, an analogue of \cite[Lemma 2.11]{CavalieriChanUlirschWise_tropstack} holds in the context $(\mathbf{RPCC}^\infty, \tau_{face}, \PP_{strict})$ and therefore the existence of a strict and surjective atlas $\Ubar_{g,n}\rightarrow \calMbar_{g,n}^{trop}$ that is representable by extended cone complexes already implies that the diagonal of $\calMbar_{g,n}^{trop}$ is representable by extended cone complexes. \end{proof}

The proof of Proposition \ref{prop_Mgn=Mbargn} in particular shows that $\calMbar_{g,n}^{trop}$ is naturally equivalent to the canonical extension of $\calM_{g,n}^{trop}$. 

\subsection{Clutching and gluing}

In \cite{CavalieriChanUlirschWise_tropstack}*{Section 6} the authors have introduced a new incarnation of the tropicalization map for the moduli space of curves as a smooth and surjective logarithmic morphism
\begin{equation*}
\trop_{g,n}\colon\calM_{g,n}^{log}\longrightarrow \calM_{g,n}^{trop} \ .
\end{equation*}
For this to make sense one has to use the theory of \emph{Artin fans} (see \cite{AbramovichWise_invariance, AbramovichChenMarcusWise_boundedness, Abramovichetal_logsurvey, Ulirsch_nonArchArtin}) in order to lift the moduli stack $\calM_{g,n}^{trop}$ to a stack over the category of logarithmic schemes. Indeed in \cite{CavalieriChanUlirschWise_tropstack}*{Section~6}, an equivalence of two categories is demonstrated between the category of cone stacks and the category of Artin fans, i.e. logarithmic algebraic stacks admitting a strict \'etale cover by a disjoint union of quotients $[A/T]$ of affine toric varieties $A$ by their dense torus. In a slight abuse of notation, we denoted by $\calM_{g,n}^{trop}$ both the cone stack described above and the associated Artin fan which is representable by algebraic stack with a logarithmic structure. 

\begin{definition}
Let $S$ be a pointed logarithmic scheme. A \emph{family of extended tropical curves} over $S$ is a collection $\Gamma_q$ of tropical curves, indexed by the geometric points $q$ of $S$, with edge lengths in $\Mbar_{S,q}$ such that, whenever $t$ is a geometric point of $S$ that specializes to $q$, then the tropical curve $\Gamma_t$ is obtained from $\Gamma_q$ by endowing the underlying graph $G_q$ of $\Gamma_q$ with the edge length 
\begin{equation*}
d\colon E(G_q)\longrightarrow \Mbar_{S,q}\longrightarrow \Mbar_{S,t}
\end{equation*}
and contracting all edges for which this edge length is zero.
\end{definition}

Denote by $\calMbar_{g,n}^{trop}$ the fibered category over $\mathbf{LSch}^\infty$ whose fiber over pointed logarithmic scheme $S$ is the groupoid of families of stable tropical curves over $S$ of genus $g$ with $n$ marked legs. 

\begin{proposition}\label{prop_Mgnbarlog=Mgnbar}
The fibered category $\calMbar_{g,n}^{trop}$ is representable by an algebraic stack with a pointed logarithmic structure that is logarithmically \'etale over $k$. 
\end{proposition}

\begin{proof}
This is an immediate consequence of Proposition \ref{prop_Mgn=Mbargn}, since it shows that $\calMbar_{g,n}^{trop}$ is nothing but the canonical extension of $\calM_{g,n}^{trop}$ for every geometric point of a pointed logarithmic scheme. Then, since $\calM_{g,n}^{trop}$ is an algebraic stack with a logarithmic structure that is logarithmically \'etale over $k$ by \cite{CavalieriChanUlirschWise_tropstack}*{Theorem 3}, the claim follows. 
\end{proof}

\begin{definition}
Let $S$ be a pointed logarithmic scheme whose underlying scheme is a point. Given a logarithmic curve $X$ over $S$, the \emph{dual tropical curve} $\Gamma_X$ is the extended tropical curve consisting of:
\begin{enumerate}[(i)]
\item one vertex $v$ for each irreducible component $X_v$ of $X$, with vertex weighted $h(v)$ the genus of the normalization of $X$; 
\item a leg $l_i$ incident to the vertex $v$ for each marked point $x_i$ on $X_v$; and
\item for each node $x_e$ of $X$ connecting two component $X_v$ and $X_{v'}$ with logarithmic equation $\alpha+\beta=\delta_e$ (see Definition \ref{def_logcurve} above) an edge $e$ connecting the two vertices $v, v'$ of length $d(e)=\delta_e\in \Mbar_S$.
\end{enumerate}
\end{definition}

In particular, we have a natural strict, smooth, and surjective  \emph{tropicalization morphism} 
\begin{equation*}
\trop_{g,n}\colon\calMbar_{g,n}^{log}\longrightarrow \calMbar_{g,n}^{trop}
\end{equation*}
that is given by associating to a family of logarithmic curves over a pointed logarithmic scheme $S$ the family of dual tropical curves over $S$.

\begin{theorem}\label{thm_tropclutch&glue}
There are natural clutching maps 
\begin{equation*}
\calMbar_{g,n+1}^{trop}\times \calMbar_{g',n'+1}^{trop}\longrightarrow \calMbar_{g+g',n+n'}^{trop}
\end{equation*}
and gluing maps
\begin{equation*}
\calMbar_{g-1, n+2}^{trop}\longrightarrow\calMbar_{g,n}^{trop}
\end{equation*}
(as morphisms of extended cone stacks) that make the induced diagrams (in the category of pointed logarithmic stacks)
\begin{equation*}
\begin{CD}
\calMbar_{g,n+1}^{log}\times \calMbar_{g',n'+1}^{log}@>\trop>>\calMbar_{g,n+1}^{trop}\times \calMbar_{g',n'+1}^{trop}\\
@VVV @VVV\\
\calMbar_{g+g',n+n'}^{log} @>\trop>> \calMbar_{g+g',n+n'}^{trop}
\end{CD}
\qquad\qquad
\begin{CD}
\calMbar_{g-1, n+2}^{log}@>\trop>>\calMbar_{g-1, n+2}^{trop}\\
@VVV @VVV\\
\calMbar_{g,n}^{log} @>\trop>>\calMbar_{g,n}^{trop}
\end{CD}
\end{equation*}
commute. 
\end{theorem}

\begin{proof}
Define the clutching map 
\begin{equation*}
\calMbar_{g,n+1}^{trop}\times\calMbar_{g',n'}^{trop}\longrightarrow \calMbar_{g+g',n+n'}^{trop}
\end{equation*}
as the unique map whose restriction to an extended rational polyhedral cone $\sigmabar$ is given by the association 
\begin{equation*}
\big([\Gamma,l_1,\ldots,l_n,\star],[\Gamma',l_1',\ldots, l_{n'}',\bullet]\big)\longmapsto \big[\Gamma\sqcup_{\star\sim\bullet}\Gamma'\big] \ .
\end{equation*}
The tropical curve $\Gamma\sqcup_{\star\sim\bullet}\Gamma'$ is defined by taking the amalgamated sum of the underlying graphs of $\Gamma$ and $\Gamma'$ over the legs $\star$ and $\bullet$ and endowing the resulting graph with the generalized edge length
\begin{equation*}
d(e)=\left\{\begin{array}{cl}
d_\Gamma(e) &\quad\textrm{ if } e\in E(\Gamma)  \\
d_{\Gamma'}(e) &\quad\textrm{ if } e\in E(\Gamma')  \\
 \infty &\quad\textrm{ if } e=\{\star\sim\bullet\}\\ \end{array}\right .
\end{equation*}
with values in $S_\sigma^\infty$. From the explicit description of the pointed logarithmic structures in Theorem \ref{thm_logclutch&glue} we obtain that this map commutes with tropicalization.

A completely analogous construction also gives a (self)-gluing map
\begin{equation*}\begin{split}
\calMbar_{g-1,n+2}^{trop}&\longrightarrow\calMbar_{g,n}^{trop}\\
\big[\Gamma,l_1,\ldots, l_n,\star,\bullet\big]&\longmapsto \big[\Gamma/_{\star\sim\bullet}\big]
\end{split}\end{equation*}
that naturally commutes with tropicalization, again by Theorem \ref{thm_logclutch&glue} above.
\end{proof}

\bibliographystyle{alpha}
\bibliography{biblio}{}

\end{document}